\documentclass[a4paper,11pt]{article}
\usepackage{amsmath,cite}
\usepackage{amssymb,amsthm,tocvsec2}
\usepackage{amsmath}
\usepackage{sidecap}
\usepackage{wrapfig}
\usepackage{pict2e}
\usepackage{graphicx}
\usepackage{setspace}
\usepackage{float}
\usepackage{color}
\usepackage{mathrsfs}
\usepackage{epsfig,epsf,latexsym,subfig}
\usepackage{mdwlist}
\usepackage{paralist}
\usepackage{setspace}
\usepackage{enumitem}
\setlist{nolistsep}
\usepackage{indentfirst}
\usepackage{color}
\usepackage{authblk}

\usepackage{caption}

\graphicspath{ {./Nume_result/} }

\catcode`\@=11
\@addtoreset{equation}{section} 
\numberwithin{equation}{section} \numberwithin{figure}{section}
\numberwithin{table}{section}

\theoremstyle{plain}
\newtheorem{thm}{Theorem}[section]
\newtheorem{lem}{Lemma}[section]

\newtheorem{rem}{Remark}[section]

\newcommand{\bm}{\boldsymbol}

\newcommand{\be}{\begin{equation}}
\newcommand{\ee}{\end{equation}}

\newcommand{\bse}{\begin{subequations}}
\newcommand{\ese}{\end{subequations}}
\def\benl{\begin{eqnarray*}}
\def\eenl{\end{eqnarray*}}

\def\be{\bm{e}}

\def\bmu1{\bm{\mu_1}}

\newcommand{\ben}{\begin{eqnarray}}
\newcommand{\een}{\end{eqnarray}}
\newcommand{\beq}{\begin{equation}}
\newcommand{\eeq}{\end{equation}}
\newcommand{\bea}{\begin{array}}
\newcommand{\eea}{\end{array}}
\newcommand{\bef}{\begin{figure}[H]}
\newcommand{\eef}{\end{figure}}

\def\part#1{\frac{\partial #1}{\partial t}}

\title
{Interaction of a vortex induced by a rotating cylinder with a  plane}

\author[1]{Daozhi Han \thanks{djhan@iu.edu}}
\author[1]{Yifeng Hou\thanks{houyifeng1005@hotmail.com}}
\author[1]{Roger Temam\thanks{temam@indiana.edu}}
\affil[1]{Department of Mathematics \& Institute for Scientific Computing and Applied Mathematics, Indiana University at Bloomington,
47406, USA}

\date{}

\begin{document}
\maketitle

\begin{abstract}

In this article,we study theoretically and numerically the interaction of a vortex induced by a rotating cylinder  with a perpendicular plane. We show the existence of weak solutions to the swirling vortex models by using the Hopf extension method, and by an elegant contradiction argument, respectively. We demonstrate numerically that the model could produce phenomena of swirling vortex  including boundary layer pumping and two-celled vortex that are observed in potential line vortex interacting with a plane and in a tornado.

\end{abstract}


\section{Introduction}
The aim of this article is to study the interaction of a vertical line vortex with a horizontal plane as a first simplified model for tornadoes. The study is partly theoretical and partly numerical.  In the analytical part, the line vortex is replaced by a vertical rotating cylinder of small radius $\sigma$, and we show the existence of an axisymmetric weak solution to the stationary Navier-Stokes equations, when $\sigma >0$. In the numerical part, we  examine carefully the fluid flow near the plane boundary and around the cylinder.

The study of a line vortex interacting with a plane perpendicular to the vortex core, even in the simple axisymmetric setting,  is important,   as the structures of the resulting swirling vortex (sometimes exact solutions to the stationary Navier-Stokes equations) can give insight into the  dramatic phenomenon of tornado. Goldshtik \cite{Goldshtik1960} discovered a family of conical self-similar swirling vortex solutions to the axisymmetric Navier-Stokes equations resulting from a vertical potential line vortex of constant circulation interacting with an infinite orthogonal plane . However, Goldshtik's vortex solutions exist only for small Reynolds numbers ($Re < Re^\ast=5.53$), since the solutions are assumed to be bounded at the vortex axis \cite{Goldshtik1990}. Serrin partially resolved Goldshtik's paradox and showed the existence of self-similar vortex solutions for all Reynolds number,  if one allows singularity formation at the vortex core \cite{Serrin1972}. Due to the lack of boundary conditions, Serrin's solutions depend on an additional parameter that needs to be specified as an input to the physical system. The parameter amounts to specifying a boundary condition for the pressure on the plane surface. Nevertheless, Serrin shows that  these solutions have rich structures including two-celled vortex. Recently, there has been research on modifying the scaling of the velocity/radial distance dependence in Serrin's vortex solutions, based on radar data observation \cite{BDSS2014}.

The major controversy in Serrin's idealized model of a line vortex stems from the vortex singularity at the vortex axis which serves as a source of momentum. In the present work, we regularize the line vortex by a rotating cylinder of small radius. Serrin's model of a line vortex can be viewed as an asymptotic limit of the rotating cylinder when its radius approaches zero. Moreover, as the vortex singularity is regularized, no additional physical parameters are needed in our model other than the circulation and kinematic viscosity.

Now we describe the problem of tornado-like vortex driven by  a uniform rotating cylinder of a small radius in details. 
The natural coordinates system for this problem is the cylindrical coordinates $(r, \theta, z)$. Let $(u,v, w)$ be the velocity vector in the cylindrical coordinates. Then the axisymmetric flow is governed by the following dimensional steady-state Navier-Stokes equations
\begin{align}
&u \frac{\partial u}{\partial r}+ w\frac{\partial u}{\partial z}-\frac{v^2}{r}=-\frac{1}{\rho}\frac{\partial p}{\partial r} +\nu (\frac{\partial ^2 u}{\partial r^2}+\frac{1}{r}\frac{\partial u}{\partial r}-\frac{u}{r^2}+\frac{\partial^2 u}{\partial z^2}), \label{NS1}\\
&u \frac{\partial v}{\partial r}+ w\frac{\partial v}{\partial z}+\frac{uv}{r}=\nu (\frac{\partial ^2 v}{\partial r^2}+\frac{1}{r}\frac{\partial v}{\partial r}-\frac{v}{r^2}+\frac{\partial^2 v}{\partial z^2}), \label{NS2}\\
&u \frac{\partial w}{\partial r}+ w\frac{\partial w}{\partial z}=-\frac{1}{\rho}\frac{\partial p}{\partial z} +\nu (\frac{\partial ^2 w}{\partial r^2}+\frac{1}{r}\frac{\partial w}{\partial r}+\frac{\partial^2 w}{\partial z^2}), \label{NS3}\\
&\frac{\partial(ru) }{\partial r}+\frac{\partial (rw)}{\partial z}=0,  \label{NS4}
\end{align}
where $\rho$ is the density, and $\nu$ is the kinematic viscosity. The domain is defined as $\Omega=\{(r, \theta, z) \quad | \quad r\geq \sigma, \theta \in (0, 2\pi), 0\leq z \leq L\}$. We decompose the boundary $\partial \Omega$ into several parts such that $\partial \Omega=\Gamma_i \cup \Gamma_l \cup \Gamma_u$ with $\Gamma_i=\{r\geq \sigma, z=0\}$, $\Gamma_l=\{r= \sigma, 0\leq z \leq L\}$ and $\Gamma_u=\{r\geq \sigma, z=L\}$. The boundary conditions for the flow that we would like to impose are
\begin{align}\label{NS_bc}
&(u, v, w)|_{\Gamma_i}=0, \quad (u, v, w)|_{\Gamma_l}=(0, \frac{\gamma}{2\pi \sigma}, 0), \quad (\frac{\partial u}{\partial z}, \frac{\partial v}{\partial z}, w)|_{\Gamma_u}=0.
\end{align}
Here $\gamma$ is the circulation prescribed on the surface of the  cylinder.
Recall that a potential vortex line is given by $u=w=0, v=\frac{\gamma}{2\pi r}=-\phi^\prime(r)$ with the potential $\phi(r)=-\frac{\gamma}{2\pi } \ln r$. The boundary condition $v|_{\Gamma_l}$ reflects  the condition given by the line vortex. The boundary conditions on $\Gamma_u$ are no-penetration  plus free-slip boundary conditions. In the development below, we will impose additional boundary conditions either $(u,v,w) \rightarrow (0, 0, 0)$ as $r\rightarrow \infty$ when the domain is unlimited in the radial direction, or $(u, v, \frac{\partial w}{\partial r})=(0, 0, 0) \text{ on } r=R$ when the domain is limited to $\sigma \leq r \leq R$ as below in Section 3, with a large $R>0$. These boundary conditions are chosen for convenience and are commonly adopted for the study of tornadoes.

Our model \eqref{NS1}-\eqref{NS4} with the boundary conditions \eqref{NS_bc} is a regularization of the problem of the interaction of a potential line vortex with an orthogonal plane, in which the  axis of singularity of the potential line vortex is replaced by a rotating cylinder of small radius. Indeed, the solution to the system without the plane boundaries $\Gamma_i, \Gamma_u$ is given by $u=w=0, v=\frac{\gamma}{2\pi r}$ where $r\geq \sigma$. Hence this solution can be viewed as a truncation of the potential line vortex away from its singular core. In our numerical simulation, we will also explore different combinations of the rotation speed of the cylinder, i.e., other than the uniform rotation $v|_{\Gamma_l}=\frac{\gamma}{2\pi \sigma}$. In this regard, the rotating cylinder can be viewed as an approximation (simplification) of a tornado funnel which is least understood to date.

The system \eqref{NS1}-\eqref{NS4} is the stationary axisymmetric Navier-Stokes equations posed in a Lipschitz domain with mixed $L^p$ Dirichlet boundary data (piecewise-constant) and homogeneous Neumann boundary data \eqref{NS_bc}. The existence of solutions in appropriate weak sense is a challenging problem. For the physical problem of vortex interacting with an orthogonal plane, the major difficulty with the model \eqref{NS1}-\eqref{NS_bc} comes from the mismatch of the boundary data at the corner $r=\sigma, z=0$. In Sec. 2, we smooth out the discontinuity and employ a boundary condition inspired by Hopf's construction of background flow, cf. \cite{Hopf1951} (see also \cite{Ladyzhenskaya1969,Temam1977}). With this slightly modified boundary condition, we are able to establish the existence of weak solutions as in the  classical weak theory for Navier-Stokes equations. The method of Hopf's extension relies on that the boundary data has some regularity (for instance, in $H^{\frac12}$). On the other hand, the well-posedness and regularity of the Stokes problem with $L^p$ Dirichlet boundary condition in Lipschtiz domain has been studied intensively, for instance, in \cite{FKV1988, Shen1995, MiTa2001}. Based on the regularity property of the Stokes problem, we employ an elegant contradiction argument to show existence of weak solutions to the system \eqref{NS1}-\eqref{NS4} under the original boundary condition \eqref{NS_bc}. See Sec. 3 for details. Such a contradiction argument was first introduced by Leray when studying the existence of solutions to the stationary Navier-Stokes equations; it  was explored by many authors later on, cf.  \cite{Ladyzhenskaya1969, Amick1984, KPR2013}. More recently, the authors in \cite{KPR2015} utilized this technique to solve the Leray's conjecture.

In Sec. 4, we present the numerical study of the swirling vortex induced by a rotating cylinder
in the presence of a perpendicular static plane. We observe that as we shrink the radii of the cylinder, the flow configuration tends to the one observed in the case of potential line vortex interacting with a plane. In particular, the unique phenomenon of boundary layer pumpiing associated with rotating fluid in the vicinity of solid wall is verified in the numerical simulation of the model. Moreover, we demonstrate that rich flow pattern, including two-celled vortex, can be produced by exploiting different combinations of rotation speed of the cylinder. The rotating cylinder model can serve as a new platform for the study of flow in a tornado funnel that is less understood. 

\section{Existence of solutions via Hopf's Extension} 
One immediately realizes that the boundary conditions at the``corner'' are mismatched. In order to be able to show the existence of a solution and have a continuous boundary condition, we could change slightly the boundary conditions and adopt the following boundary condition inspired  by the construction of Hopf \cite{Hopf1951}, see also \cite{Temam1977}. Define a cut-off function $\xi_\epsilon(r)$ such that
\begin{equation}
\xi_\epsilon(r)=\left\{
\begin{aligned}
&1, \quad \sigma \leq r \leq \sigma+\delta^2,  \\
&\epsilon \ln(\frac{\delta}{r-\sigma}), \quad \sigma+\delta^2 \leq r \leq \sigma+ \delta, \\
& 0, \quad r \geq \sigma+\delta,
\end{aligned}
\right.
\end{equation}
where $0< \epsilon <1$ is to be determined and $\delta=e^{-1/\epsilon}$. It is clear that $\xi_\epsilon$ is continuous and piecewise smooth. In particular, one has
\begin{equation}
\xi_\epsilon^\prime(r)=\left\{
\begin{aligned}
&0, \quad \sigma \leq r \leq \sigma+\delta^2,  \\
&-\frac{\epsilon}{r-\sigma}, \quad \sigma+\delta^2 \leq r \leq \sigma+ \delta, \\
& 0, \quad r \geq \sigma+\delta.
\end{aligned}
\right.
\end{equation}
One can further regularize $\xi_\epsilon$ by Friedrich's mollifier or by averaging such that the regularized cut-off function $\theta_\epsilon$ satisfies $\theta_\epsilon \in C^2(\sigma, \infty)$ and $supp (\theta_\epsilon) \subset (\sigma, \sigma+\delta+\frac{\delta^2}{2})$.  Moreover, one has
\begin{equation}\label{ProTheta}
\theta_\epsilon(r)=\left\{
\begin{aligned}
&1, \quad \sigma \leq r \leq \sigma+\frac{\delta^2}{2},  \\
& 0, \quad r \geq \sigma+\delta+\frac{\delta^2}{2},
\end{aligned}
\right.
\end{equation}
and $|\theta_\epsilon^\prime(r)| \leq \frac{\epsilon}{r-\sigma}$ if $\sigma+\frac{\delta^2}{2} \leq r \leq \sigma+\delta+\frac{\delta^2}{2}$. 

Recall the stream function for the potential vortex $\phi(r)=-\frac{\gamma}{2\pi}\ln(r)$. We introduce  the background flow related to the vortex line motion, namely $(\tilde{u}, \tilde{v}, \tilde{w})=(0, -\frac{\partial (\theta_\epsilon \phi)}{\partial r}, 0)$. It is clear from Eq. \eqref{ProTheta} that 
\begin{align}
& (\tilde{u}, \tilde{v}, \tilde{w})|_{\Gamma_l}=(0, \frac{\gamma}{2\pi \sigma}, 0), \\
& (\tilde{u}, \tilde{v}, \tilde{w})|=(0, 0, 0), \quad {\text{at }} z=0, r \geq \sigma +\delta+\frac{\delta^2}{2}. 
\end{align}
Moreover, one has  $\text{div}(\tilde{u}, \tilde{v}, \tilde{w})=\frac{1}{r}\frac{\partial (r \tilde{u})}{\partial r}+\frac{1}{r}\frac{\partial \tilde{v}}{\partial \theta}+\frac{\partial \tilde{w}}{\partial z}=0$ by construction. Note that the background field $(\tilde{u}, \tilde{v}, \tilde{w})$ depends on $\epsilon$. Here we do not mark this dependence for convenience.

With the help of the background flow $(\tilde{u}, \tilde{v}, \tilde{w})$, we replace the boundary conditions \eqref{NS_bc} by
\begin{align}\label{NS_bc_new}
(u, v, w)|_{\Gamma_i \cup \Gamma_l}=(\tilde{u}, \tilde{v}, \tilde{w})|_{\Gamma_i \cup \Gamma_l}, \quad (\frac{\partial u}{\partial z}, \frac{\partial v}{\partial z}, w)|_{\Gamma_u}=0.
\end{align}
We note that the difference between the boundary conditions \eqref{NS_bc} and \eqref{NS_bc_new} is restricted to the plane $z=0, \sigma < r < \sigma +\delta +\frac{\delta^2}{2}$ of small area.

In Hopf's work (see \cite{Hopf1951, Temam1977}) one then determines $\epsilon$ such that the stationary NS system \eqref{NS1}-\eqref{NS4} equipped with the boundary conditions \eqref{NS_bc_new} is ``coercive'' so that one can prove the existence of a weak solution (see below). We proceed in a formal way for the moment. Define a new velocity $(u^\prime, v^\prime, w^\prime)$ via 
\begin{align*}
u=u^\prime+\tilde{u}=u^\prime, \quad v=v^\prime+\tilde{v}, \quad w=w^\prime+\tilde{w}=w^\prime.
\end{align*}
One sees that the new velocity $(u^\prime, v^\prime, w^\prime)$ satisfies the following equations and boundary conditions, dropping the primes for notational simplicity,
\begin{align}
&u \frac{\partial u}{\partial r}+ w\frac{\partial u}{\partial z}-\frac{(v+\tilde{v})^2}{r}=-\frac{1}{\rho}\frac{\partial p}{\partial r} +\nu \Delta_{r, z} u, \label{NS1_new}\\
&u \frac{\partial v}{\partial r}+u \frac{\partial \tilde{v}}{\partial r}+ w\frac{\partial v}{\partial z}+\frac{u(v+\tilde{v})}{r}=\nu (\Delta_{r, z} v + \Delta_{r, z} \tilde{v}), \label{NS2_new}\\
&u \frac{\partial w}{\partial r}+ w\frac{\partial w}{\partial z}=-\frac{1}{\rho}\frac{\partial p}{\partial z} +\nu (\frac{\partial ^2 w}{\partial r^2}+\frac{1}{r}\frac{\partial w}{\partial r}+\frac{\partial^2 w}{\partial z^2}), \label{NS3_new}\\
&\frac{\partial(ru) }{\partial r}+\frac{\partial (rw)}{\partial z}=0,  \label{NS4_new} \\
& (u, v, w)|_{\Gamma_i \cup \Gamma_l}=0, \quad (\frac{\partial u}{\partial z}, \frac{\partial v}{\partial z}, w)|_{\Gamma_u}=0, \quad (u,v,w) \rightarrow 0, \text{ as } r\rightarrow \infty, \label{NS5_new}
\end{align}
where $\Delta_{r, z}$ is the operator $\Delta_{r, z}:= \frac{\partial^2}{\partial r^2}+ \frac{1}{r} \frac{\partial}{\partial r}-\frac{1}{r^2}+\frac{\partial^2}{\partial z^2}$.

To study the system \eqref{NS1_new}-\eqref{NS5_new}, it is convenient to work with the domain $\Omega^\prime:=\{(r,z), \sigma \leq r < \infty, 0\leq z \leq L\}$. We introduce the space
\begin{align*}
V_1=\Big\{v \big| \int_{\Omega^\prime} [(\frac{\partial v}{\partial r})^2+(\frac{\partial v}{\partial z})^2]\, rdrdz < \infty, v|_{z=0}=v|_{r=\sigma}=0\Big\}.
\end{align*}
One can verify that $V_1$ is a Hilbert space when equipped with the $H^1$-inner product:
\begin{align*}
(v, v_1)_{V_1}=\int_{\Omega^\prime} \big(\frac{\partial v}{\partial r} \frac{\partial v_1}{\partial r}+\frac{\partial v}{\partial z} \frac{\partial v_1}{\partial z}\big) \, rdrdz.
\end{align*}
We then define another space
\begin{align*}
{\bf{V}}_2=\Big\{(u, w) \in V_1 \times V_1 \big| \frac{\partial (ru)}{\partial r}+\frac{\partial (rw)}{\partial z}=0, w|_{z=L}=0\Big\}.
\end{align*}
The weak formulation of the system \eqref{NS1_new}-\eqref{NS5_new} reads as follows: find $({\bf{U}}, v):=(u, w, v) \in {\bf{V}_2} \times V_1 $ such that for any $({\bf{U}}_1, v_1):=(u_1, w_1, v_1) \in {\bf{V}}_2 \times V_1$, there holds
\begin{align}
&\int_{\Omega^\prime}({\bf{U}} \cdot \nabla){\bf{U}} \cdot {\bf{U}}_1 \, rdrdz -\int_{\Omega^\prime}\frac{(v+\tilde{v})^2}{r} u_1 \, rdrdz= -\nu \int_{\Omega^\prime} \nabla {\bf{U}}: \nabla {\bf{U}}_1\, rdrdz -\nu \int_{\Omega} \frac{u u_1}{r^2}\, rdrdz, \label{WNS1} \\
&\int_{\Omega^\prime} ({\bf{U}} \cdot \nabla v) v_1\, rdrdz+\int_{\Omega^\prime} u \frac{\partial \tilde{v}}{\partial r} v_1\, rdrdz + \int_{\Omega^\prime} \frac{u(v+\tilde{v})}{r} v_1\, rdrdz=-\nu \int_{\Omega^\prime} \nabla v \cdot \nabla v_1\, rdrdz   \label{WNS2}\\
& \quad \quad -\nu \int_{\Omega^\prime} \nabla \tilde{v} \cdot \nabla v_1\, rdrdz -\nu  \int_{\Omega} \frac{(v+\tilde{v}) v_1}{r^2}\, rdrdz. \nonumber
\end{align}

Now we derive certain  energy estimates of  weak solutions to the system \eqref{WNS1}-\eqref{WNS2}. Taking the test functions ${\bf U}_1={\bf U}$ in Eq. \eqref{WNS1} and $v_1=v$ in Eq. \eqref{WNS2}, summing up the resulting equations,  one obtains
\begin{align}\label{EnerEs1}
&\nu \int_{\Omega^\prime} r(\frac{\partial u}{\partial r})^2 +\frac{u^2}{r}+r (\frac{\partial u}{\partial z})^2 \, dr  dz + \nu \int_{\Omega^\prime} r(\frac{\partial v}{\partial r})^2 +\frac{v^2}{r}+r (\frac{\partial v}{\partial z})^2 \, dr  dz   \\ \nonumber 
& + \nu \int_{\Omega^\prime} r(\frac{\partial w}{\partial r})^2+r (\frac{\partial w}{\partial z})^2 \, dr  dz =\int_{\Omega^\prime} uv \tilde{v}+ u \tilde{v}^2 \, dr  dz -\nu \int_{\Omega^\prime} r \frac{\partial \tilde{v}}{\partial r} \frac{\partial v}{\partial r} +\frac{\tilde{v} v}{r} \, dr  dz.
\end{align}
For the last three terms in Eq. \eqref{EnerEs1}, one applies the Cauchy-Schwarz inequality as follows

\begin{align*}
\big|\int_{\Omega^\prime}u \tilde{v}^2 -\nu  r \frac{\partial \tilde{v}}{\partial r} \frac{\partial v}{\partial r} +\nu \frac{\tilde{v} v}{r} \, dr  dz \big| &\leq \frac{1}{2\nu} \int_{\Omega^\prime} r \tilde{v}^4 \, dr  dz +\frac{\nu}{2} \int_{\Omega^\prime} \frac{u^2}{r} \, dr  dz \\
 & +\frac{\nu}{2} \int_{\Omega^\prime}r(\frac{\partial \tilde{v}}{\partial r})^2 +\frac{\tilde{v}^2}{r} \, dr  dz+\frac{\nu}{2} \int_{\Omega^\prime}r(\frac{\partial v}{\partial r})^2 +\frac{v^2}{r} \, dr  dz.
\end{align*}
Then  Eq. \eqref{EnerEs1} becomes
\begin{align}\label{EnerEs2}
&\nu \int_{\Omega^\prime} r(\frac{\partial u}{\partial r})^2 +\frac{u^2}{2r}+r (\frac{\partial u}{\partial z})^2 \, dr  dz + \nu \int_{\Omega^\prime} \frac{r}{2}(\frac{\partial v}{\partial r})^2 +\frac{v^2}{2r}+r (\frac{\partial v}{\partial z})^2 \, dr  dz   \\ \nonumber 
& + \nu \int_{\Omega^\prime} r(\frac{\partial w}{\partial r})^2+r (\frac{\partial w}{\partial z})^2 \, dr  dz \leq \int_{\Omega^\prime} uv \tilde{v}\, dr  dz +\frac{\nu}{2} \int_{\Omega^\prime}r(\frac{\partial \tilde{v}}{\partial r})^2 +\frac{\tilde{v}^2}{r} \, dr  dz+ \frac{1}{2\nu} \int_{\Omega^\prime} r \tilde{v}^4 \, dr  dz.
\end{align}

To derive an a priori estimate on $(\mathbf{U}, v)$ from Eq. \eqref{EnerEs2}, we need to properly bound the integral $\int_{\Omega^\prime} uv \tilde{v}\, dr  dz$ appearing in the RHS of \eqref{EnerEs2} and proceed as in \cite{Hopf1951, Temam1977}. For that purpose,  we show that we can choose $\epsilon$ such that $|\int_{\Omega^\prime} uv \tilde{v}\, dr  dz|$ can be arbitrarily small in some sense. One has $|\int_{\Omega^\prime} uv \tilde{v}\, dr  dz| \leq ||u/r||_{L^2} ||v\tilde{v}||_{L^2}$. We control $||v\tilde{v}||_{L^2}$ by following the proof of Lemma $1.8$, Chapter 2,  in \cite{Temam1977}. Recall that $\tilde{v}=-\frac{\partial (\theta_\epsilon \phi)}{\partial r}$. One has
 \begin{align}\label{EsBack}
 |\tilde{v}| \leq C (\frac{\epsilon}{r-\sigma}+|\phi^\prime(r)|). 
 \end{align}
 It then follows from Hardy's inequality and H\"{o}lder's inequality that 
\begin{align*}
||v\tilde{v}||_{L^2} &\leq \epsilon||\frac{v}{r-\sigma}||_{L^2}+ \Big(\int_{(z, \theta)} \big(\int_{\sigma}^{\sigma+\delta +\frac{\delta^2}{2}} v^2|\phi^\prime|^2 r\, dr\big)dzd\theta\Big)^{\frac{1}{2}}  \\
&\leq C\epsilon||\frac{\partial v}{\partial r}||_{L^2}+\Big(\int_\sigma^{\sigma+\delta+\frac{\delta^2}{2}}(\phi^\prime)^4 r^2\, dr \Big)^{\frac{1}{4}} \Big(\int_{(z, \theta)} \big(\int_{\sigma}^{\sigma+\delta +\frac{\delta^2}{2}} v^4\, dr\big)^{\frac{1}{2}}dzd\theta\Big)^{\frac{1}{2}} \\
&\leq  C\epsilon||\frac{\partial v}{\partial r}||_{L^2} +C \big(\frac{\delta}{\sigma^2}\big)^{\frac{1}{4}} \Big(\int_{(z, \theta)} \big(\int_{\sigma}^{\sigma+\delta +\frac{\delta^2}{2}} v^4\, dr\big)^{\frac{1}{2}}dzd\theta\Big)^{\frac{1}{2}},
\end{align*}
where $C$ is a constant independent of $\sigma, \epsilon$ and $\delta$ which may be different at each occurrence.
Now one has
\begin{align*}
 \Big(\int_{\sigma}^{\sigma+\delta +\frac{\delta^2}{2}} v^4\, dr\Big)^{\frac{1}{2}} &\leq \Big(\int_{\sigma}^{\sigma+\delta +\frac{\delta^2}{2}} \big(\int_\sigma^r (\frac{\partial v (\tau, z)}{\partial \tau})^2 \, d\tau \big)^2(r-\sigma)^2 \, dr\Big) \\
 & \leq \int_{\sigma}^{\sigma+\delta +\frac{\delta^2}{2}} (\frac{\partial v}{\partial r})^2\, dr \Big(\int_{\sigma}^{\sigma+\delta +\frac{\delta^2}{2}} (r-\sigma)^2\, dr \Big)^{\frac{1}{2}}\\
 &\leq C \delta^3  \int_{\sigma}^{\sigma+\delta +\frac{\delta^2}{2}} (\frac{\partial v}{\partial r})^2\, dr \\
 &\leq C\frac{\delta^3}{\sigma}\int_{\sigma}^{\sigma+\delta +\frac{\delta^2}{2}} r(\frac{\partial v}{\partial r})^2\, dr.
\end{align*}
Collecting the above two inequalities, one concludes that 
\begin{align}\label{KeyEs}
||v\tilde{v}||_{L^2} &\leq (C\epsilon+C \frac{\delta^\frac{7}{4}}{{\sigma}}) ||\frac{\partial v}{\partial r}||_{L^2}.
\end{align}
Hence,
\begin{align*}
|\int_{\Omega^\prime} uv \tilde{v}\, dr  dz| \leq \frac{\nu}{4}||u/r||_{L^2}^2 + \frac{C}{\nu}(\epsilon+ \frac{\delta^\frac{7}{4}}{{\sigma}}) ||\frac{\partial v}{\partial r}||_{L^2}^2.
\end{align*}

Recall that $\delta=e^{-1/\epsilon}$. Let $\epsilon$ be sufficiently small so that $C (\epsilon+\frac{\delta^\frac{7}{4}}{{\sigma}}) \leq \frac{\nu^2}{4}$. Then the inequality \eqref{EnerEs2} becomes
\begin{align}\label{EnerEs3}
&\nu \int_{\Omega^\prime} r(\frac{\partial u}{\partial r})^2 +\frac{u^2}{4r}+r (\frac{\partial u}{\partial z})^2 \, dr  dz + \nu \int_{\Omega^\prime} \frac{r}{4}(\frac{\partial v}{\partial r})^2 +\frac{v^2}{2r}+r (\frac{\partial v}{\partial z})^2 \, dr  dz   \\ \nonumber 
& + \nu \int_{\Omega^\prime} r(\frac{\partial w}{\partial r})^2+r (\frac{\partial w}{\partial z})^2 \, dr  dz \leq \frac{\nu}{2} \int_{\Omega^\prime}r(\frac{\partial \tilde{v}}{\partial r})^2 +\frac{\tilde{v}^2}{r} \, dr  dz+ \frac{1}{2\nu} \int_{\Omega^\prime} r \tilde{v}^4 \, dr  dz.
\end{align}
One can thus obtain bounds for $||\nabla(u, v, w)||_{L^2}$, hence bounds for $||(u, v, w)||_{H^1}$ by Poincar\'{e}'s inequality. This would allow the passage to limit in the procedure of Galerkin approximation in any compact sub-domain of $\Omega^\prime$, cf. \cite{Temam1977} for details. In conclusion, we  have proven the following theorem.
\begin{thm}\label{ThmEx}
The steady axisymmetric Navier-Stokes system \eqref{NS1}--\eqref{NS4} equipped with the boundary conditions \eqref{NS_bc_new} admits a weak solution in the sense that $(u, v-\tilde{v}, w)$ satisfies the weak formulation \eqref{WNS1}-\eqref{WNS2}. 
\end{thm}

\section{Existence of solutions without extension}
We now study the axisymmetric Navier-Stokes system \eqref{NS1}--\eqref{NS4} with the original boundary conditions \eqref{NS_bc},  and establish the existence of solutions weaker than that in Theorem  \ref{ThmEx}. For technical convenience, we introduce an artificial boundary at far field $r=R$ for a large constant $R>0$, denoted by $\Gamma_3$. On $\Gamma_3$, the following boundary conditions are imposed
\begin{align}\label{bc_g3}
(u, v, w)=(0, 0, 0), \text{ on } r=R.
\end{align}
The artificial boundary condition \eqref{bc_g3} will be used in the numerical simulations below as well. 

\begin{rem}
The boundary condition $v=0$ at $\Gamma_3$ ensures that the linear axisymmetric Stokes system with mixed Dirichlet-Neumann boundary  condition can be converted to a Stokes system with purely Dirichlet boundary condition via reflection in the $z$ direction, see Eqs. \eqref{ne_vs} and \eqref{ne_bc} below. Hence the  theory of Stokes system with $L^p$ boundary data in a Lipschitz domain is applicable. We point out that if the free-slip boundary condition is imposed on $\Gamma_3$, i.e.,   $\frac{\partial v}{\partial r}|_{r=R}=0$, one would have to make an even extension in the radial direction in order to obtain a Stokes problem with Dirichlet boundary condition. However,  the Eq. \eqref{ne_vs} is not invariant with even extension in the $r$ direction. In this case, one has to study a Stokes system with mixed Dirichlet-Neumann boundary condition (or more precisely, an elliptic equation with mixed boundary condition). For the latter case, we refer to \cite{BMMW2010, TOB2013, Lanzani2008, MaSh2011} and references therein for the latest developments.
\end{rem}

Dropping the nonlinear terms, we first consider the axisymmetric Stokes system
\begin{align}
&-\frac{1}{\rho}\frac{\partial p_s}{\partial r} +\nu (\frac{\partial ^2 u_s}{\partial r^2}+\frac{1}{r}\frac{\partial u_s}{\partial r}-\frac{u_s}{r^2}+\frac{\partial^2 u_s}{\partial z^2})=0, \label{S1}\\
&\nu (\frac{\partial ^2 v_s}{\partial r^2}+\frac{1}{r}\frac{\partial v_s}{\partial r}-\frac{v_s}{r^2}+\frac{\partial^2 v_s}{\partial z^2})=0, \label{S2}\\
&-\frac{1}{\rho}\frac{\partial p_s}{\partial z} +\nu (\frac{\partial ^2 w_s}{\partial r^2}+\frac{1}{r}\frac{\partial w_s}{\partial r}+\frac{\partial^2 w_s}{\partial z^2})=0, \label{S3}\\
&\frac{\partial(ru_s) }{\partial r}+\frac{\partial (rw_s)}{\partial z}=0,  \label{S4}
\end{align}
with the same boundary conditions as in \eqref{NS_bc} and \eqref{bc_g3}. In particular, $v_s$ satisfies
\begin{align}
&\frac{\partial ^2 v_s}{\partial r^2}+\frac{1}{r}\frac{\partial v_s}{\partial r}-\frac{v_s}{r^2}+\frac{\partial^2 v_s}{\partial z^2}=0, \label{e_vs}\\
&v_s|_{z=0}=v_s|_{r=R}=0, v_s|_{\sigma}=\frac{\gamma}{2\pi \sigma}, \frac{\partial v_s}{\partial z}|_{z=L}=0. \label{e_bc}
\end{align}
 It is clear that $u_s=w_s=0$, $p_s=const$ and $v_s$ satisfying  Eq. \eqref{S2} equipped with  the boundary conditions \eqref{e_bc} is  solution to the system \eqref{S1}-\eqref{S4}. Now we consider an even extension of $v_s$ with respect to $z=L$, that is, we define $\tilde{v}_s$ such that 
 \begin{eqnarray*}
 \tilde{v}_s(r, z)=\left\{
 \begin{aligned}
 &v_s(r, z), \text{ for } 0\leq z \leq L, \\
 &v_s(r, 2L-z),  \text{ for } L\leq z \leq 2L,
 \end{aligned}
 \right.
 \end{eqnarray*}
 One can check that $\tilde{v}_s$ satisfies
 \begin{align}
&\frac{\partial ^2 \tilde{v}_s}{\partial r^2}+\frac{1}{r}\frac{\partial \tilde{v}_s}{\partial r}-\frac{\tilde{v}_s}{r^2}+\frac{\partial^2 \tilde{v}_s}{\partial z^2}=0, \label{ne_vs}\\
&\tilde{v}_s|_{z=0}=\tilde{v}_s|_{r=R}=\tilde{v}_s|_{z=2L}=0, \tilde{v}_s|_{r=\sigma}=\frac{\gamma}{2\pi \sigma}. \label{ne_bc}
\end{align}
Hence one only needs to consider the Stokes problem with $L^p$ Dirichlet boundary data in a domain with a  Lipschitz boundary. The following lemma summarizes the regularity results pertaining to Stokes system in Lipschitz domains, cf. \cite{FKV1988, Shen1995, MiTa2001}
\begin{lem}\label{lem_stokes}
Let $\Omega$ be a Lipschitz domain, $\mathbf{a}$ is the Dirichlet boundary data on $\partial \Omega$. If $\mathbf{a} \in L^2(\partial \Omega)$, then the `homogeneous Stokes equations with the Dirichlet boundary data $\mathbf{a}$ admits a unique solution (up to a constant for the pressure) such that  $(\mathbf{u}_s, p_s) \in [\mathbf{H}^{\frac{1}{2}}(\Omega) \cap W_{loc}^{2, 1}(\Omega)] \times W_{loc}^{1, 1}(\Omega)$. If $\mathbf{a} \in L^p(\Omega)$ for $p\in[2, \infty]$, then $\mathbf{u}_s \in W^{1/p, p}(\Omega)$.
\end{lem}
Here ``loc'' means any sub-domain of $\Omega$ that is of positive distance from the boundary $\partial \Omega$.
Lemma \ref{lem_stokes} guarantees the existence of the solution $u_s=w_s=0$, $p_s=const$ and $v_s$ to the Stokes system \eqref{S1}-\eqref{S4} under the boundary conditions \eqref{NS_bc} and \eqref{bc_g3}.

We now proceed as in Section 2.  Define a new velocity $(u^\prime, v^\prime, w^\prime)$ via 
\begin{align*}
u=u^\prime+u_s=u^\prime, \quad v=v^\prime+v_s, \quad w=w^\prime+w_s=w^\prime.
\end{align*}
One sees that the velocity $(u^\prime, v^\prime, w^\prime)$ satisfies the following equations and boundary conditions, dropping the primes for notational simplicity,
\begin{align}
&u \frac{\partial u}{\partial r}+ w\frac{\partial u}{\partial z}-\frac{(v+v_s)^2}{r}=-\frac{1}{\rho}\frac{\partial p}{\partial r} +\nu \Delta_{r, z} u, \label{NS_C1}\\
&u \frac{\partial v}{\partial r}+u \frac{\partial v_s}{\partial r}+ w\frac{\partial v}{\partial z}+w\frac{\partial v_s}{\partial z}+\frac{u(v+v_s)}{r}=\nu \Delta_{r, z} v , \label{NS_C2}\\
&u \frac{\partial w}{\partial r}+ w\frac{\partial w}{\partial z}=-\frac{1}{\rho}\frac{\partial p}{\partial z} +\nu (\frac{\partial ^2 w}{\partial r^2}+\frac{1}{r}\frac{\partial w}{\partial r}+\frac{\partial^2 w}{\partial z^2}), \label{NS_C3}\\
&\frac{\partial(ru) }{\partial r}+\frac{\partial (rw)}{\partial z}=0,  \label{NS_C4} \\
& (u, v, w)|_{\Gamma_i \cup \Gamma_l}=0, \quad (\frac{\partial u}{\partial z}, \frac{\partial v}{\partial z}, w)|_{\Gamma_u}=0, \quad (u, v, w)|_{\Gamma_3}=0, \label{NS_C5}
\end{align}
where we recall the operator  $\Delta_{r, z}= \frac{\partial^2}{\partial r^2}+ \frac{1}{r} \frac{\partial}{\partial r}-\frac{1}{r^2}+\frac{\partial^2}{\partial z^2}$.

\begin{rem}
We note that $u=v=w=0$ is not a solution to the system \eqref{NS_C1}-\eqref{NS_C5}. Otherwise, the system \eqref{NS_C1}-\eqref{NS_C5}  reduces to
\begin{equation}\label{HypoE}
\left\{
\begin{aligned}
&-\frac{v_s^2}{r}=-\frac{1}{\rho} \frac{\partial p}{\partial r}, \\
&0=-\frac{1}{\rho}\frac{\partial p}{\partial z}.
\end{aligned}
\right.
\end{equation}
Hence $p$ would be independent of $z$, and by the first equation in \eqref{HypoE}, $v_s$ would depend only on $r$.  The boundary condition $v_s|_{z=0}=0$ would then imply that  $v_s=0$ and this contradicts the boundary condition $v_s|_{r=\sigma}=\frac{\gamma}{2\pi \sigma}$ in \eqref{e_bc}.
\end{rem}

The weak formulation for the system  \eqref{NS_C1}-\eqref{NS_C5} can be defined analogously as in Section 2. With abuse of notation, 
we define the domain $\Omega^\prime:=\{(r,z), \sigma \leq r < R, 0\leq z \leq L\}$, and we modify the function space 
\begin{align*}
&V_1=\Big\{u \big| \int_{\Omega^\prime} [(\frac{\partial u}{\partial r})^2+(\frac{\partial u}{\partial z})^2]\, rdrdz < \infty, u|_{z=0}=u|_{r=\sigma}=u|_{r=R}=0\Big\}, \\
&W_1=\Big\{w \big| \int_{\Omega^\prime} [(\frac{\partial w}{\partial r})^2+(\frac{\partial w}{\partial z})^2]\, rdrdz < \infty, w|_{z=0}=w|_{r=\sigma}=w|_{z=L}=0\Big\}, \\
&{\bf{V}}_2=\Big\{(u, w) \in V_1 \times W_1 \big| \frac{\partial (ru)}{\partial r}+\frac{\partial (rw)}{\partial z}=0\Big\}.
\end{align*}
One can verify that the function spaces $V_1, W_1, \bf{V}_2$ are Hilbert spaces equipped with inner product, for instance,
\begin{align*}
(v, v_1)_{H_1}=\int_{\Omega^\prime} \big(\frac{\partial v}{\partial r} \frac{\partial v_1}{\partial r}+\frac{\partial v}{\partial z} \frac{\partial v_1}{\partial z}\big) \, rdrdz.
\end{align*}
The norm induced by the inner product,  which is equivalent to the $H^1$ norm ,  will be simply denoted by $|| \cdot ||$

The weak formulation for the system \eqref{NS_C1}-\eqref{NS_C5} reads as follows: find $({\bf{U}}, v):=(u, w, v) \in {\bf{V}_2} \times V_1 $ such that for any $(\tilde{{\bf{U}}}, \tilde{v}):=(u_1, w_1, v_1) \in {\bf{V}}_2 \times V_1$, there holds
\begin{align}
&\int_{\Omega^\prime}({\bf{U}} \cdot \nabla){\bf{U}} \cdot \tilde{{\bf{U}}} \, rdrdz -\int_{\Omega^\prime}\frac{(v+v_s)^2}{r} \tilde{u} \, rdrdz= -\nu \int_{\Omega^\prime} \nabla {\bf{U}}: \nabla \tilde{{\bf{U}}}\, rdrdz -\nu \int_{\Omega} \frac{u \tilde{u}}{r^2}\, rdrdz, \label{WNS1_new} \\
&\int_{\Omega^\prime} ({\bf{U}} \cdot \nabla v) \tilde{v}\, rdrdz-\int_{\Omega^\prime} v_s [u  \frac{\partial \tilde{v}}{\partial r} ( r)+w\frac{\partial \tilde{v}}{\partial z} ]\, rdrdz + \int_{\Omega^\prime} \frac{u(v+v_s)}{r} \tilde{v}\, rdrdz  \label{WNS2_new}\\
&=-\nu \int_{\Omega^\prime} \nabla v \cdot \nabla \tilde{v}\, rdrdz     -\nu  \int_{\Omega} \frac{(v+v_s) \tilde{v}}{r^2}\, rdrdz. \nonumber
\end{align}

We will use a contradiction argument to show  that the system \eqref{NS_C1}-\eqref{NS_C5} admits a weak solution in the sense of Eqs. \eqref{WNS1_new}-\eqref{WNS2_new}. Such a contradiction argument was first introduced by Leray when studying the existence of solutions to the stationary Navier-Stokes equations; it  was explored by many authors later on, cf.  \cite{Ladyzhenskaya1969, Amick1984, KPR2013, KPR2015}. As in the classical theory of Navier-Stokes equations, the Eqs. \eqref{WNS1_new}-\eqref{WNS2_new} can be written in the functional form $A({\bf{U}}, v)+ B({\bf{U}}, v)=0$ where the operator $A$ is the Stokes operator such that $A^{-1}B$ is continuous and compact, cf. \cite{Ladyzhenskaya1969, Temam1977} for details. The Leray-Schauder fixed point theorem states that if the solutions to the equation $A({\bf{U}_\lambda}, v_\lambda)+ \lambda B({\bf{U}}_\lambda, v_\lambda)=0$ are uniformly bounded in $\lambda$ for all $\lambda \in [0,1]$ such that the solutions exist, then there exists a solution to the equation $A({\bf{U}}, v)+ B({\bf{U}}, v)=0$. We arrive then at the following theorem.

\begin{thm}\label{ThmOut}
The steady axisymmetric Navier-Stokes system \eqref{NS1}--\eqref{NS4} equipped with the boundary conditions \eqref{NS_bc} and \eqref{bc_g3} admits a weak solution in the sense that $(u, v-v_s, w)$ satisfies the weak formulation Eqs. \eqref{WNS1_new}-\eqref{WNS2_new}.
\end{thm}
\begin{proof}
We argue by contradiction. Suppose that there are no $({\bf{U}}, v):=(u, w, v) \in {\bf{V}_2} \times V_1 $ such that the Eqs. \eqref{WNS1_new}-\eqref{WNS2_new} are satisfied. The Leray-Schauder fixed point theorem implies that there exist sequences $\{\lambda_k \}_{k=1}^{\infty} \subset [0, 1]$ and $({\bf{U}}_k, v_k)\in {\bf{V}_2} \times V_1 $ such that $\forall (\tilde{{\bf{U}}}, \tilde{v}) \in {\bf{V}}_2 \times V_1$ there hold
\begin{align}
&\lambda_k \int_{\Omega^\prime}({\bf{U}_k} \cdot \nabla){\bf{U}_k} \cdot \tilde{{\bf{U}}} \, rdrdz - \lambda_k \int_{\Omega^\prime}\frac{(v_k+v_s)^2}{r} \tilde{u} \, rdrdz= -\nu \int_{\Omega^\prime} \nabla {\bf{U}_k}: \nabla \tilde{{\bf{U}}}\, rdrdz  \label{WNS1_k} \\
&-\nu \int_{\Omega} \frac{u_k \tilde{u}}{r^2}\, rdrdz, \nonumber \\
&\lambda_k \int_{\Omega^\prime} ({\bf{U}_k} \cdot \nabla v_k) \tilde{v}\, rdrdz-\lambda_k \int_{\Omega^\prime} v_s [u_k  \frac{\partial \tilde{v}}{\partial r} +w_k \frac{\partial  \tilde{v}}{\partial z} ]\, rdrdz + \lambda_k \int_{\Omega^\prime} \frac{u_k(v_k+v_s)}{r} \tilde{v}\, rdrdz  \label{WNS2_k}\\
&=-\nu \int_{\Omega^\prime} \nabla v_k \cdot \nabla \tilde{v}\, rdrdz     -\nu  \int_{\Omega} \frac{(v_k+v_s) \tilde{v}}{r^2}\, rdrdz. \nonumber\\
&\lim_{k\rightarrow \infty} \lambda_k =\lambda_0 \in [0,1], \quad \lim_{k\rightarrow \infty} J_k=\lim_{k\rightarrow \infty} ||({\bf{U}_k}, v_k)||= \infty. \label{WNS3_k}
\end{align}
It is clear that for $\lambda_k$ small enough the solutions to Eqs. \eqref{WNS1_k}--\eqref{WNS2_k} exist. We first show that there is a subsequence of $\{\lambda_k \}_{k=1}^{\infty}$ whose limit is strictly  positive.

Define $\overline{\bf{U}}_k=J_k^{-1}{\bf{U}_k}, \overline{v}_k=J_k^{-1} v_k$. It follows that $||(\overline{\bf{U}}_k, \overline{v}_k) ||=1$. We now take the test functions $\tilde{{\bf{U}}}=J_k^{-2} {\bf{U}_k}$ in Eq. \eqref{WNS1_k} and $\tilde{v}= J_k^{-2} v_k$ in Eq. \eqref{WNS2_k}. Summing up the resulting equations gives
\begin{align}\label{Energy_con}
\nu +\nu \int_{\Omega^\prime} \frac{\overline{u}_k^2+\overline{v}_k^2}{r^2}\, rdrdz=\lambda_k \int_{\Omega^\prime} v_s \big(\frac{\partial \overline{v}_k}{\partial r} \overline{u}_k+\frac{\partial \overline{v}_k}{\partial z} \overline{w}_k + \frac{\overline{u}_k \overline{v}_k}{r}\big)\, rdrdz + \lambda_k J_k^{-1} \int_{\Omega^\prime} \frac{v_s^2 \overline{u}_k}{r}\, rdrdz.
\end{align}

Since $||(\overline{\bf{U}}_k, \overline{v}_k) ||=1$, there exists a subsequence, still denoted by $(\overline{\bf{U}}_k, \overline{v}_k)$, such that $(\overline{\bf{U}}_k, \overline{v}_k) \rightarrow ({\bf{U}}_\ast, {v}_\ast)$ weakly in ${\bf{V}_2} \times V_1$ for some $({\bf{U}}^\ast, {v}_\ast) \in {\bf{V}_2} \times V_1$. The Sobolev embedding implies that $(\overline{\bf{U}}_k, \overline{v}_k) \rightarrow ({\bf{U}}_\ast, {v}_\ast)$ strongly in ${\bf{L}^r} (\Omega^\prime)\times L^r(\Omega^\prime)$ for $r \in[1, \infty)$. These compactness results enable one to pass to the limit $k \rightarrow \infty$ in the identity \eqref{Energy_con}. One finds
\begin{align}\label{Energy_con_new}
\nu +\nu \int_{\Omega^\prime} \frac{{u}_\ast^2+{v}_\ast^2}{r^2}\, rdrdz=\lambda_0 \int_{\Omega^\prime} v_s \big(\frac{\partial {v}_\ast}{\partial r} {u}_\ast+\frac{\partial {v}_\ast}{\partial z} {w}_\ast + \frac{{u}_\ast {v}_\ast}{r}\big)\, rdrdz \, rdrdz.
\end{align}
In particular, Eq. \eqref{Energy_con_new} implies that $\lambda_0>0$. Hence one can assume that the sequence $\{\lambda_k \}_{k=1}^{\infty}$ is strictly positive.

Now we introduce the rescaled viscosity $\nu_k= (\lambda_k J_k)^{-1} \nu$,  and recall that $\overline{u}_k=J_k^{-1}u_k, \overline{v}_k=J_k^{-1}v_k, \overline{w}_k=J_k^{-1}w_k$. With these new variables, Eq. \eqref{WNS2_k} can be written as
\begin{align}
&\int_{\Omega^\prime} (\overline{\bf{U}}_k \cdot \nabla \overline{v}_k) \tilde{v}\, rdrdz-J_k^{-1} \int_{\Omega^\prime} v_s [\overline{u}_k  \frac{\partial \tilde{v}}{\partial r} +\overline{w}_k \frac{\partial  \tilde{v}}{\partial z} ]\, rdrdz +  \int_{\Omega^\prime} \frac{\overline{u}_k(\overline{v}_k+J_k^{-1}v_s)}{r} \tilde{v}\, rdrdz  \label{WNS2_k_new}\\
&=-\nu_k \int_{\Omega^\prime} \nabla \overline{v}_k \cdot \nabla \tilde{v}\, rdrdz     -\nu_k  \int_{\Omega} \frac{(\overline{v}_k+J_k^{-1}v_s) \tilde{v}}{r^2}\, rdrdz. \nonumber 
\end{align}
Letting $k \rightarrow \infty$ in Eq. \eqref{WNS2_k_new}, noting that $J_k \rightarrow \infty$ and $\nu_k \rightarrow 0$, one obtains that 
\begin{align}\label{limitV}
&\int_{\Omega^\prime} ({\bf{U}}_\ast \cdot \nabla {v}_\ast) \tilde{v}\, rdrdz+\int_{\Omega^\prime} \frac{{u}_\ast {v}_\ast}{r} \tilde{v}\, rdrdz=0, \forall \tilde{v} \in V_1.
\end{align}
If one could take the test function $\tilde{v}=v_s$ in Eq. \eqref{limitV}, one would find that  
\begin{align}\label{limitV2}
&\int_{\Omega^\prime} ({\bf{U}}_\ast \cdot \nabla {v}_\ast) {v}_s\, rdrdz+\int_{\Omega^\prime} \frac{{u}_\ast {v}_\ast}{r} {v}_s\, rdrdz=0.
\end{align}
This is a contradiction with the identity  
\eqref{Energy_con_new} as the left-hand side of Eq. \eqref{Energy_con_new}  is positive. Now one  approximates $v_s \in H^{\frac{1}{2}}$ by a sequence of functions in $V^1$ and take the approximate functions as the test function in Eq. \eqref{limitV}. The procedure of taking the limit in \eqref{limitV} then shows that Eq. \eqref{limitV2} is valid.  Hence one obtains a contradiction. This shows that the system \eqref{NS_C1}-\eqref{NS_C5} admits a weak solution in the sense of Eqs. \eqref{WNS1_new}-\eqref{WNS2_new}. 
\end{proof}

\begin{rem}
Formally setting $\nu=0$ in the system \eqref{NS1}--\eqref{NS4}, one obtains the inviscid axisymmetric Euler equations, that is,  Eqs. \eqref{NS1}--\eqref{NS4} with $\nu=0$. On the boundary, one imposes that the flow is parallel to the boundary
\begin{align*}
& \text{on } \Gamma_i: w^0=0; \text{on } \Gamma_l: u^0=0; \text{on } \Gamma_u: w^0=0.
\end{align*} 
It is clear that the solutions to the stationary axisymmetric Euler equations are not unique. In particular, any solid body rotation about the cylinder would be a solution. Hence we don't know the limit of the solutions to Navier-Stokes system  \eqref{NS1}--\eqref{NS4}
when $\nu \rightarrow 0$.

Taking the inviscid flow as $(u^0, v^0, w^0)= (0, \frac{\gamma}{2\pi r}, 0)$, we see that the discrepancy  between the boundary values of the inviscid flow and those of the viscous flow is only on the azimuthal component $v$ of the velocity and is restricted on the plane $z=0, r\geq \sigma$.
A formal magnitude argument reveals that the boundary layer thickness is proportional to $\nu^{\frac{1}{2}}$.
Denote the boundary layer correctors in this region by $(\phi_u, \phi_v, \phi_w)$.  With the profile $(u^0, v^0, w^0)= (0, \frac{\gamma}{2\pi r}, 0)$, one can derive the Prandtl boundary layer equations 
\begin{align}
&\phi_u \frac{\partial \phi_u}{\partial r} +\phi_w \frac{\partial \phi_u}{\partial z}- \frac{2v^0 \phi_v +\phi_v^2}{r}=\nu \frac{\partial^2 \phi_u}{\partial z^2}, \label{Prandtl1} \\
&\phi_u \frac{\partial \phi_v}{\partial r} +\phi_w \frac{\partial \phi_v}{\partial z} + \phi_u \frac{\partial v^0}{\partial r}+ \frac{ \phi_u v^0+\phi_u \phi_v}{r}=\nu \frac{\partial^2 \phi_v}{\partial z^2}, \label{Prandtl2} \\
&\frac{\partial(r\phi_u) }{\partial r}+\frac{\partial (r\phi_w)}{\partial z}=0,  \label{Prandtl3} 
\end{align}
together with the boundary conditions
\begin{align}
&(\phi_u, \phi_v, \phi_w)|_{\Gamma_i}= (0, -\frac{\gamma}{2\pi r}, 0),  \label{Prandtl4} \\
&(\phi_u, \phi_v, \phi_w)|_{\Gamma_l}= (0, 0, 0),  \label{Prandtl5} \\
&(\frac{\partial \phi_u}{\partial z}, \frac{\partial \phi_v}{\partial z}, \phi_w)|_{\Gamma_u}= (0, 0, 0).  \label{Prandtl6}
\end{align}
We see that the boundary conditions for $\phi_v$ are mismatched at the corner $r=\sigma, z=0$. Hence one would expect some sort of corner singularity for the boundary layer $\phi_u$. This is confirmed by our numerical simulations in Sec. \ref{Nure}. We want to study in a future work the nature of the corner boundary layer in the spirit of \cite{ShKe1987,  HJT2016, HHJT2016}. Of course the issue of the existence of solutions to Eqs. \eqref{Prandtl1}-\eqref{Prandtl6} relates to the longstanding problem of existence of solutions for the Prandtl boundary layer equation \cite{OlSa1997, AWXY2015, GeDo2010, SaCa1998, XiZh2004, EEn1997}, which is compounded in the present case, by the fact that the stationary inviscid solution is not unique.
\end{rem}

\section{The numerical results}\label{Nure}
In this section, we describe some numerical results for the system \eqref{NS1}-\eqref{NS4} under the boundary conditions \eqref{NS_bc}. We impose an artificial boundary conditions at $r=R$ for a large constant $R$:
\begin{align*}
(u, v, w)|_{r=R}=(0, 0, 0).
\end{align*}
We solve the nonlinear system through a Newton's iteration procedure, starting with an initial guess given by the solution of large time to the corresponding evolutionary axisymmetric Navier-Stokes equations. 

In the first set of numerical results, we fix the (rotational) Reynolds number $Re=\frac{\gamma}{\nu}=10\pi$ and vary the radius of the cylinder $\sigma$.  The  contour plots of the velocity components $u, v, w$ of different radii $\sigma=0.1, 0.05, 0.02$ are shown in Fig. \ref{FirstFig}, \ref{SecFig} and \ref{ThrFig},  respectively. These plots exhibit a single  swirling vortex near the rotating cylinder. From the contours of the azimuthal velocity in Fig. \ref{ecFig}, one observes that the contours are nearly straight away from the corner whose magnitudes are roughly inversely proportional to   the distance to the cylinder, and near the corner the contours show singular (self-similar) structures. The axial velocity in FIg. \ref{FirstFig}  reveals the updraft in the swirling vortex. And the radial velocity (Fig. \ref{ThrFig}) shows that the flow converges toward the cylinder and then flow out. As one shrinks the the radius of the rotating cylinder, the flow pattern looks more like the one induced by potential vortex.
\begin{figure}[h!]
\centering
\begin{tabular}{cc}
   \includegraphics[width=0.4\linewidth]{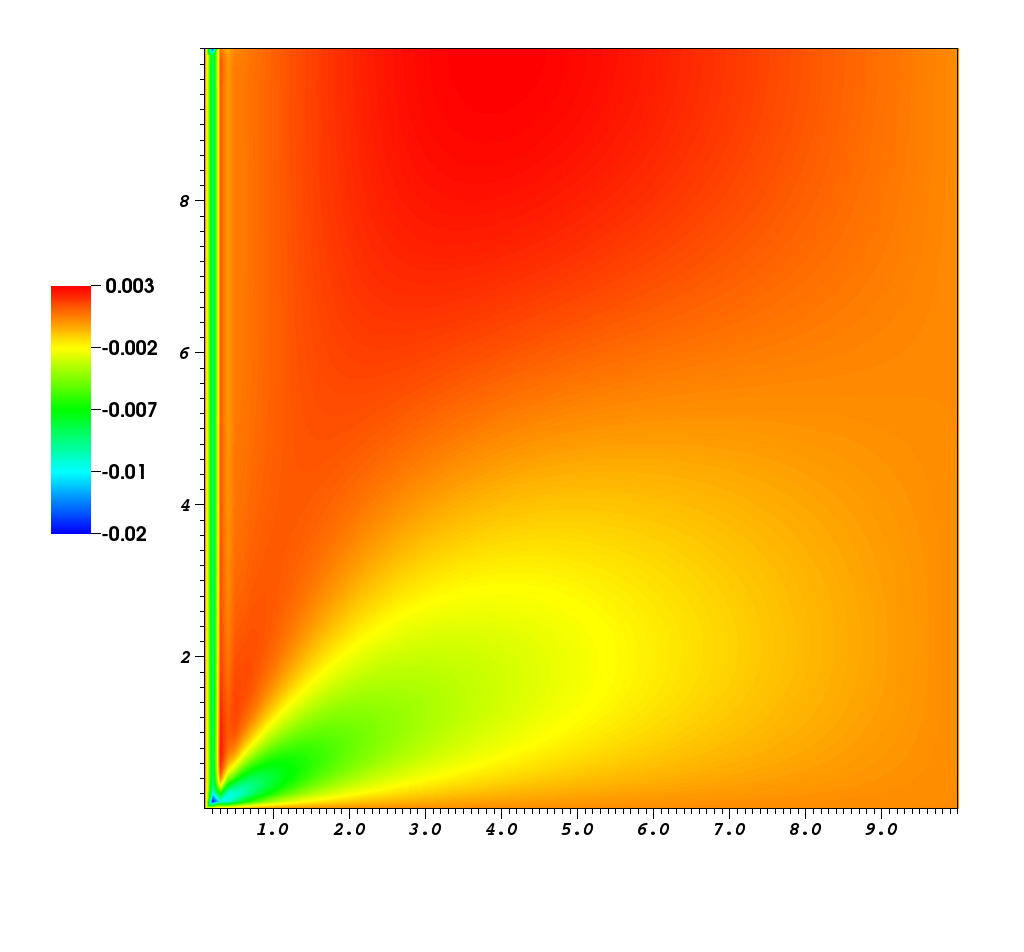}& \includegraphics[width=0.4\linewidth]{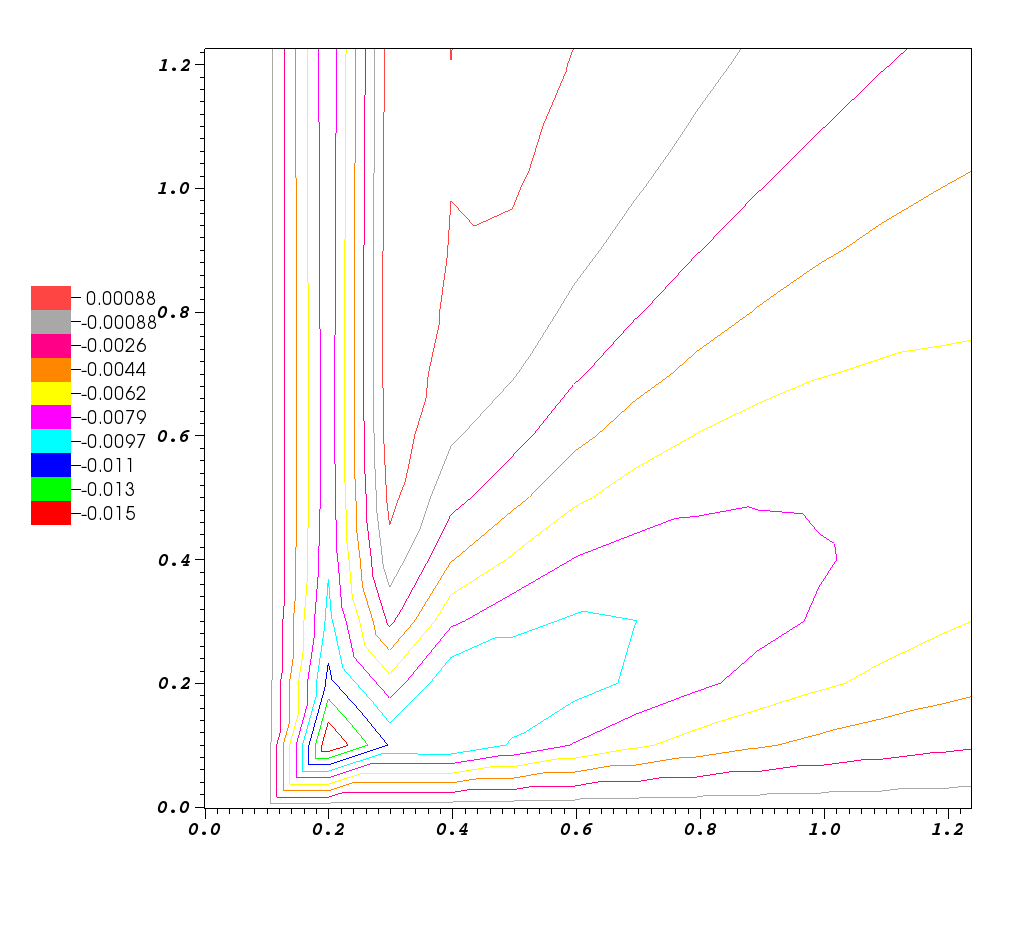}\\
   \includegraphics[width=0.4\linewidth]{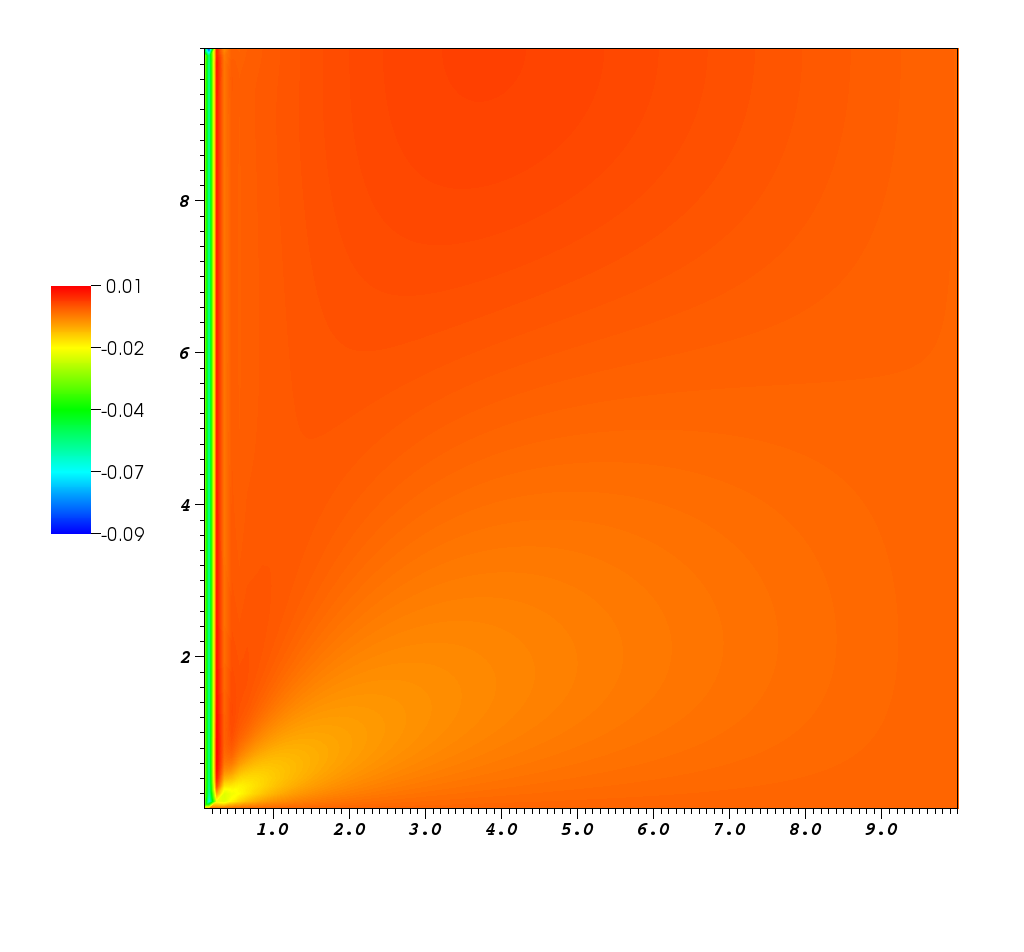}& \includegraphics[width=0.4\linewidth]{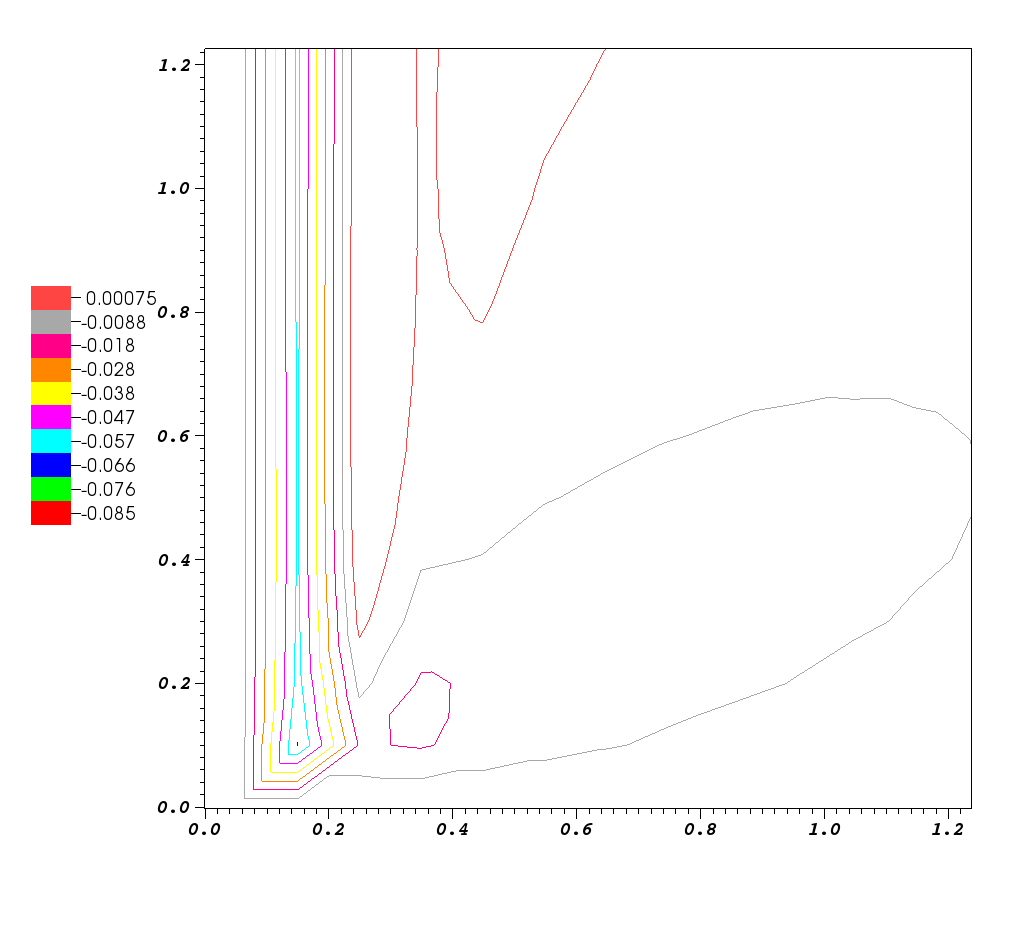}\\
  \includegraphics[width=0.4\linewidth]{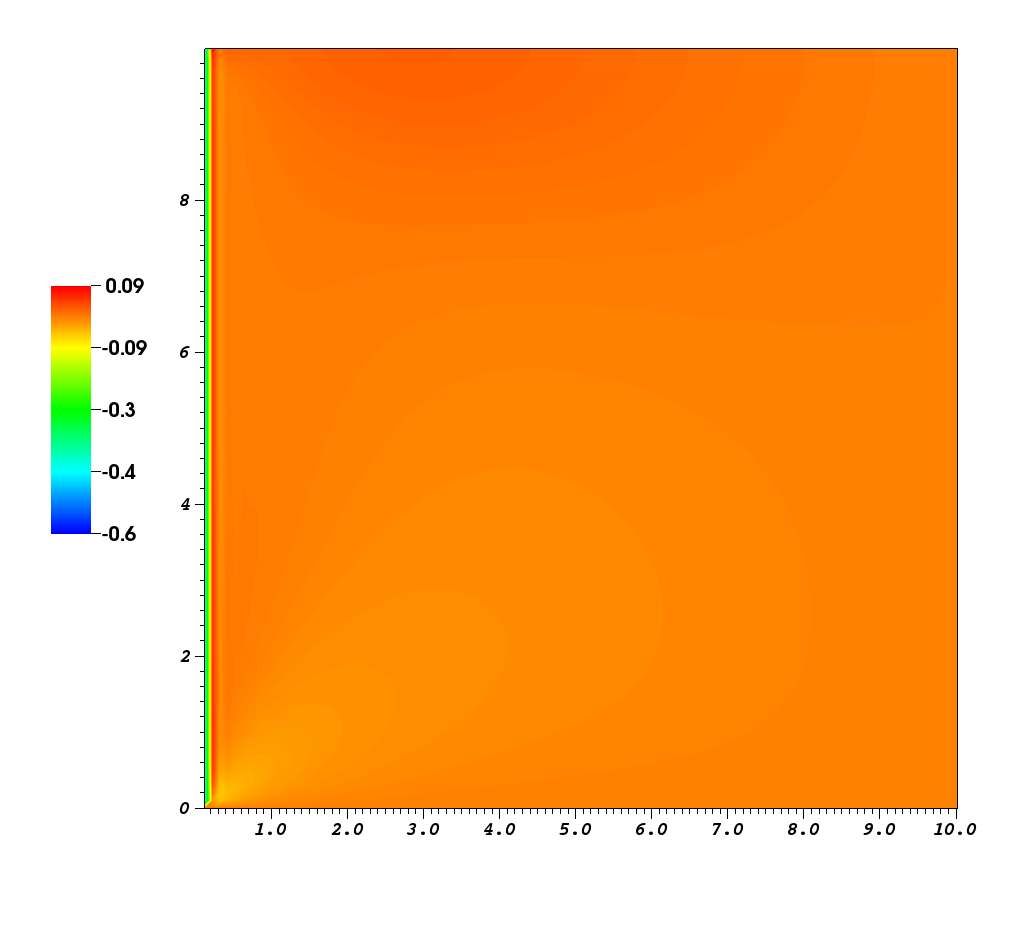}& \includegraphics[width=0.4\linewidth]{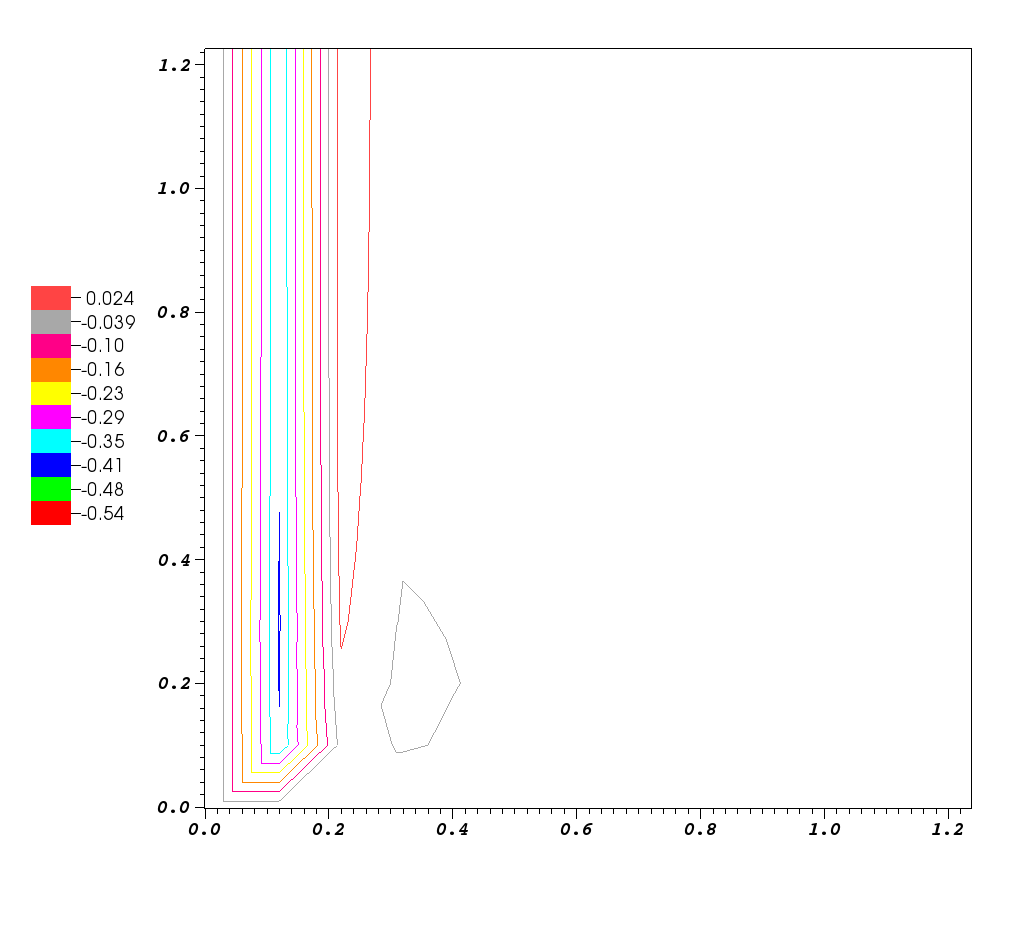}
\end{tabular}
\caption{Contour plots of the velocity component $u$ with different radii $\sigma$. From top to bottom, $\sigma=0.1, 0.05, 0.02$, respectively.  Left column, filled contour plots; right column, contour plots. The parameters are $\nu=0.01$,  $\gamma=0.1 \pi$, $Re=\frac{\gamma}{\nu}=10 \pi$.}
\label{FirstFig}
\end{figure}

\begin{figure}[h!]
\centering
\begin{tabular}{cc}
\includegraphics[width=0.4\linewidth]{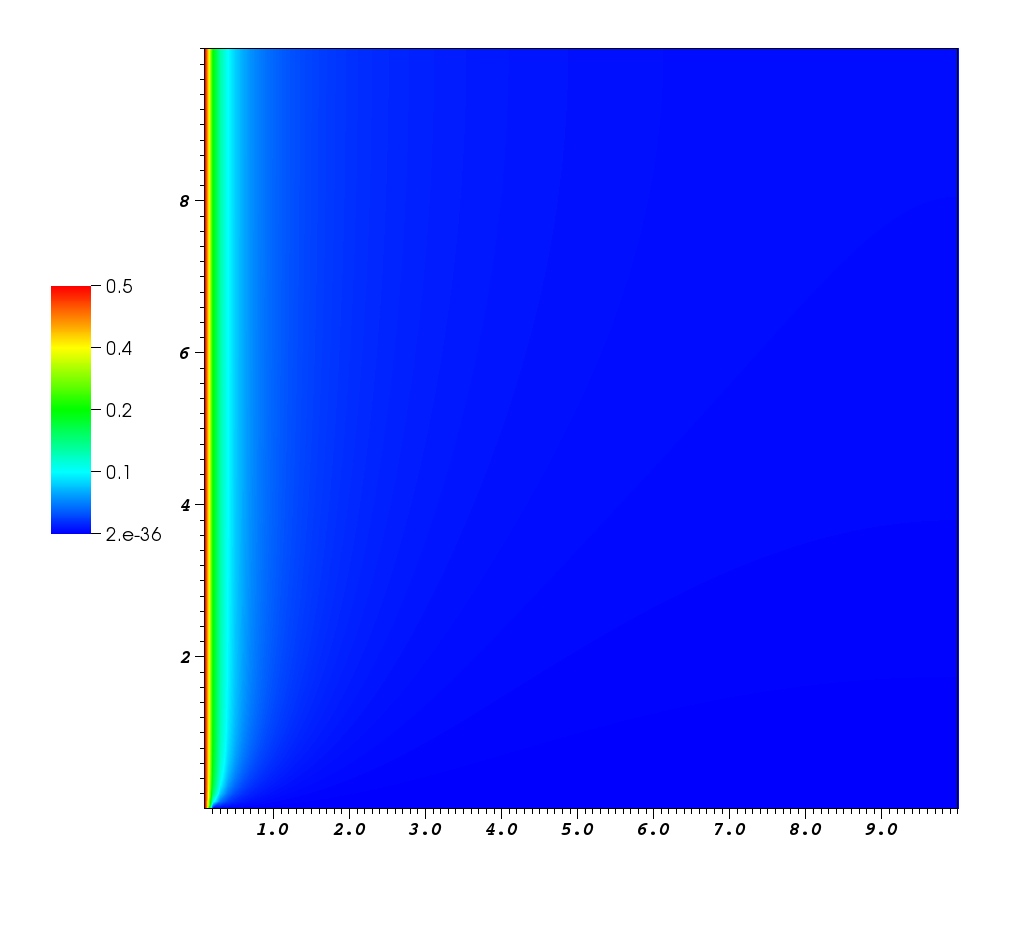}& \includegraphics[width=0.4\linewidth]{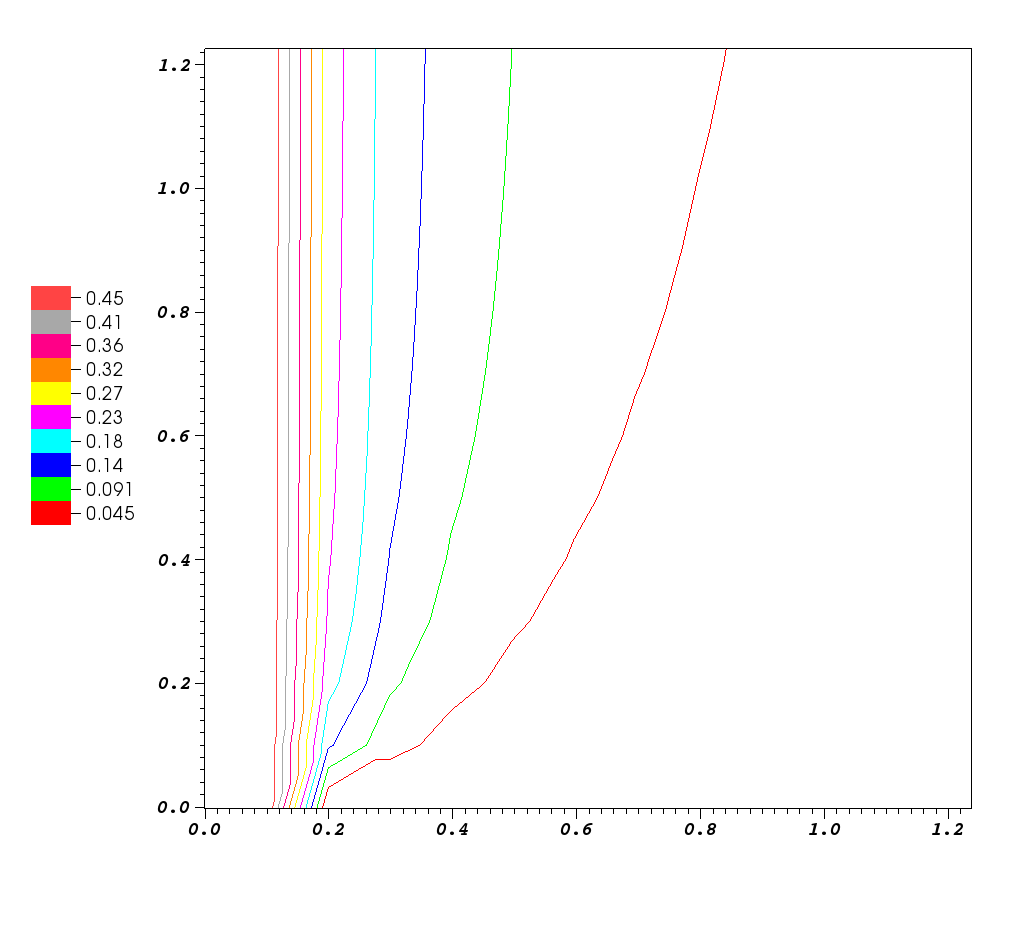}\\
   
   \includegraphics[width=0.4\linewidth]{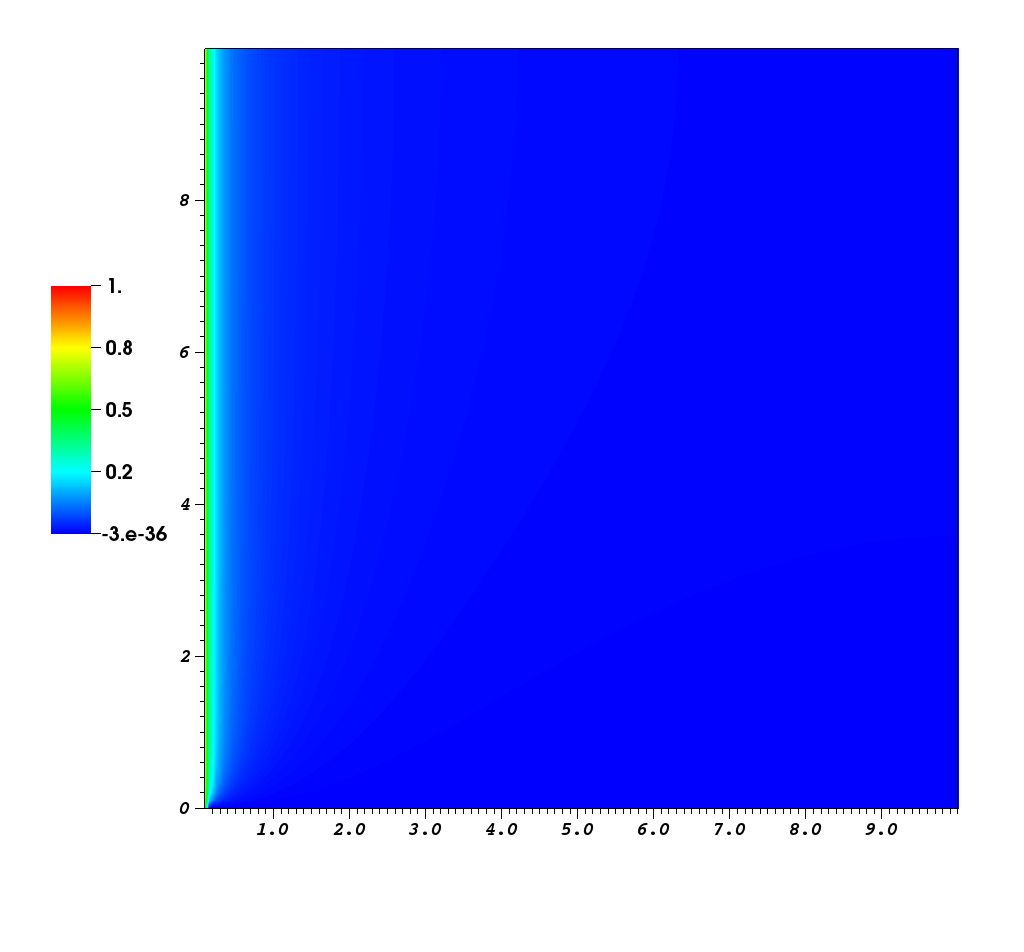}& \includegraphics[width=0.4\linewidth]{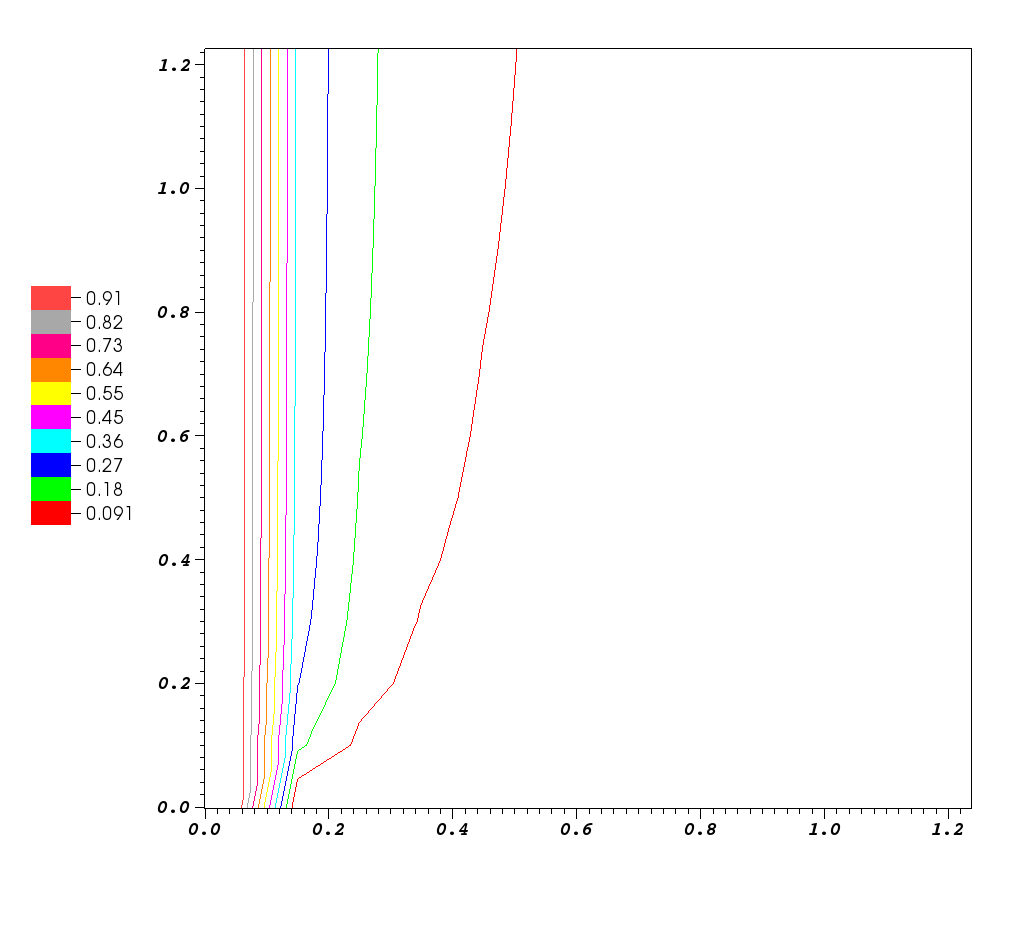}\\
    \includegraphics[width=0.4\linewidth]{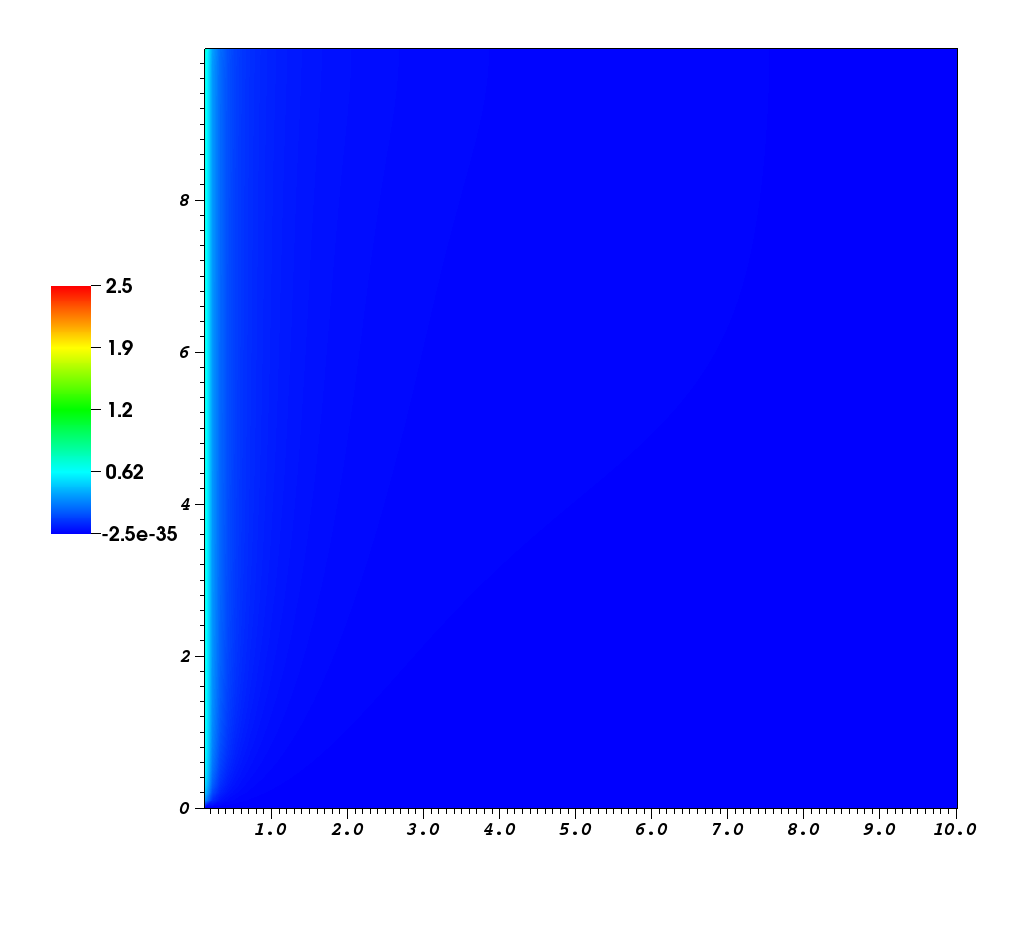}& \includegraphics[width=0.4\linewidth]{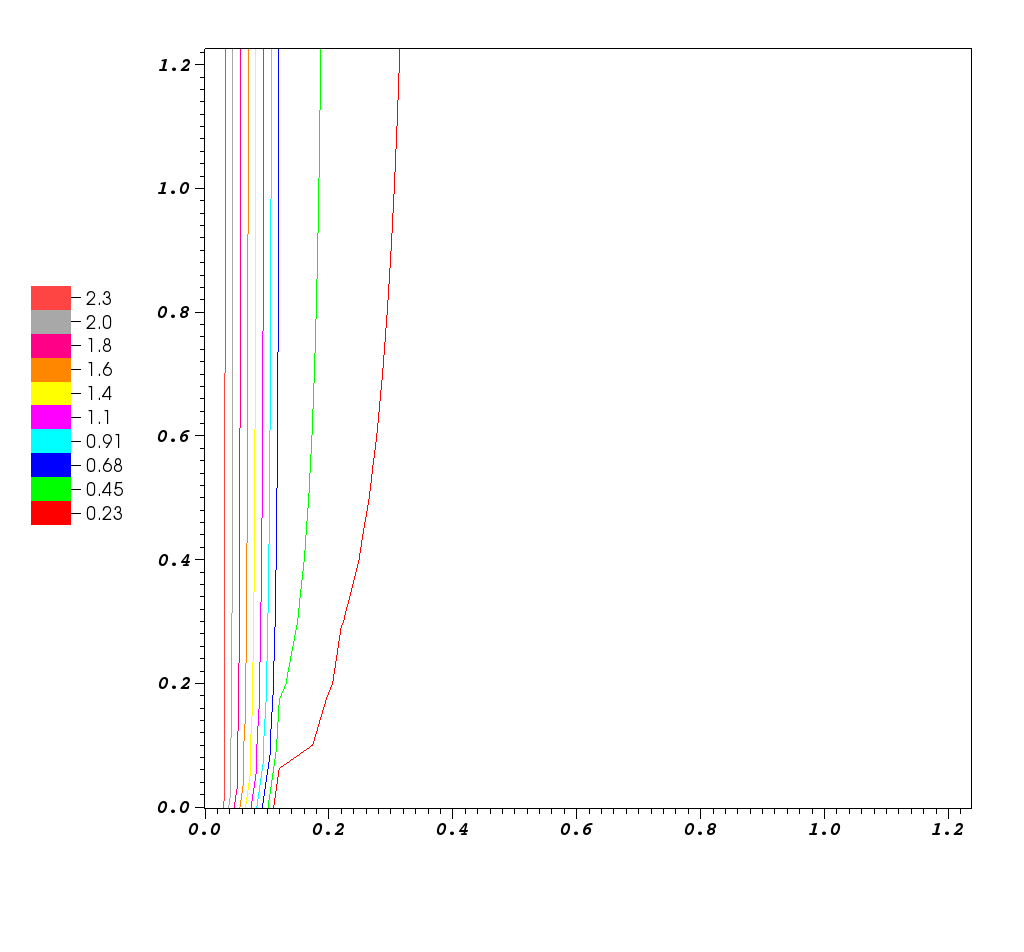} 
\end{tabular}
\caption{Contour plots of the velocity component $v$ with different radii $\sigma$. From top to bottom, $\sigma=0.1, 0.05, 0.02$, respectively.  Left column, filled contour plots; right column, contour plots. The parameters are $\nu=0.01$,  $\gamma=0.1 \pi$, $Re=\frac{\gamma}{\nu}=10 \pi$.}
\label{SecFig}
\end{figure}

\begin{figure}[h!]
\centering
\begin{tabular}{cc}
\includegraphics[width=0.4\linewidth]{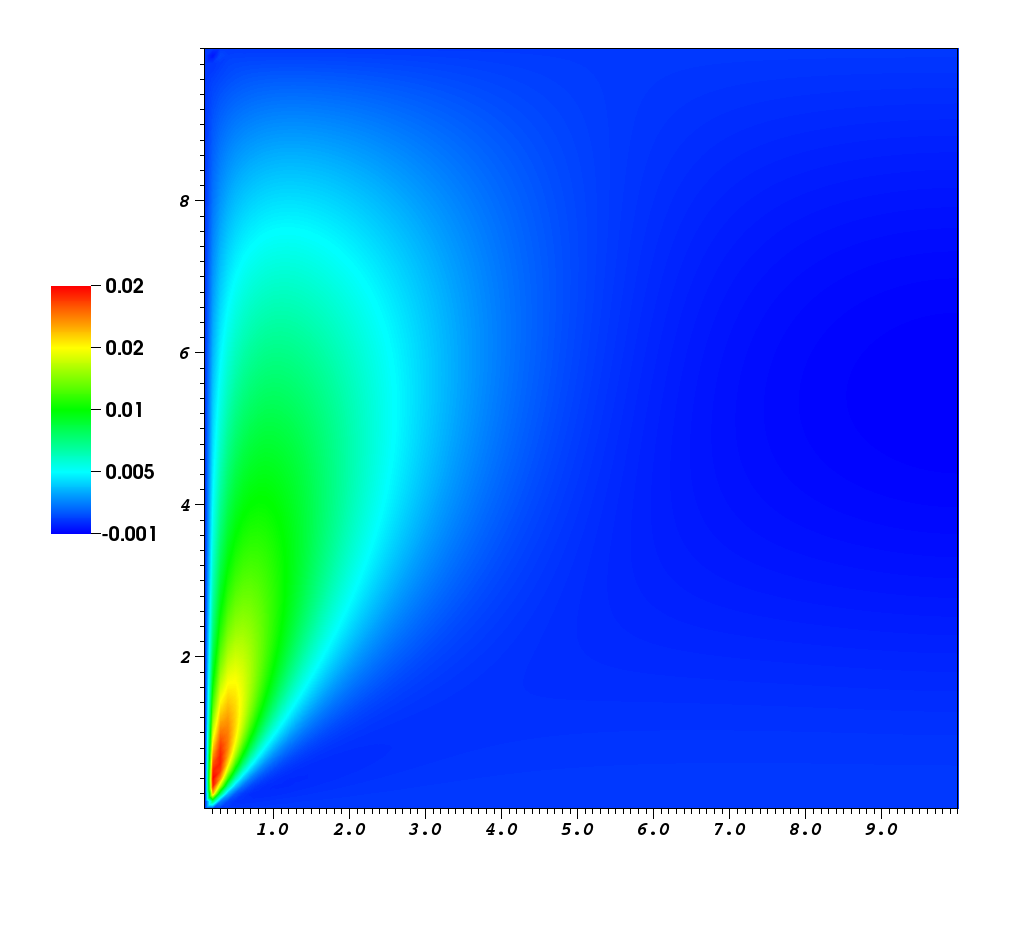}& \includegraphics[width=0.4\linewidth]{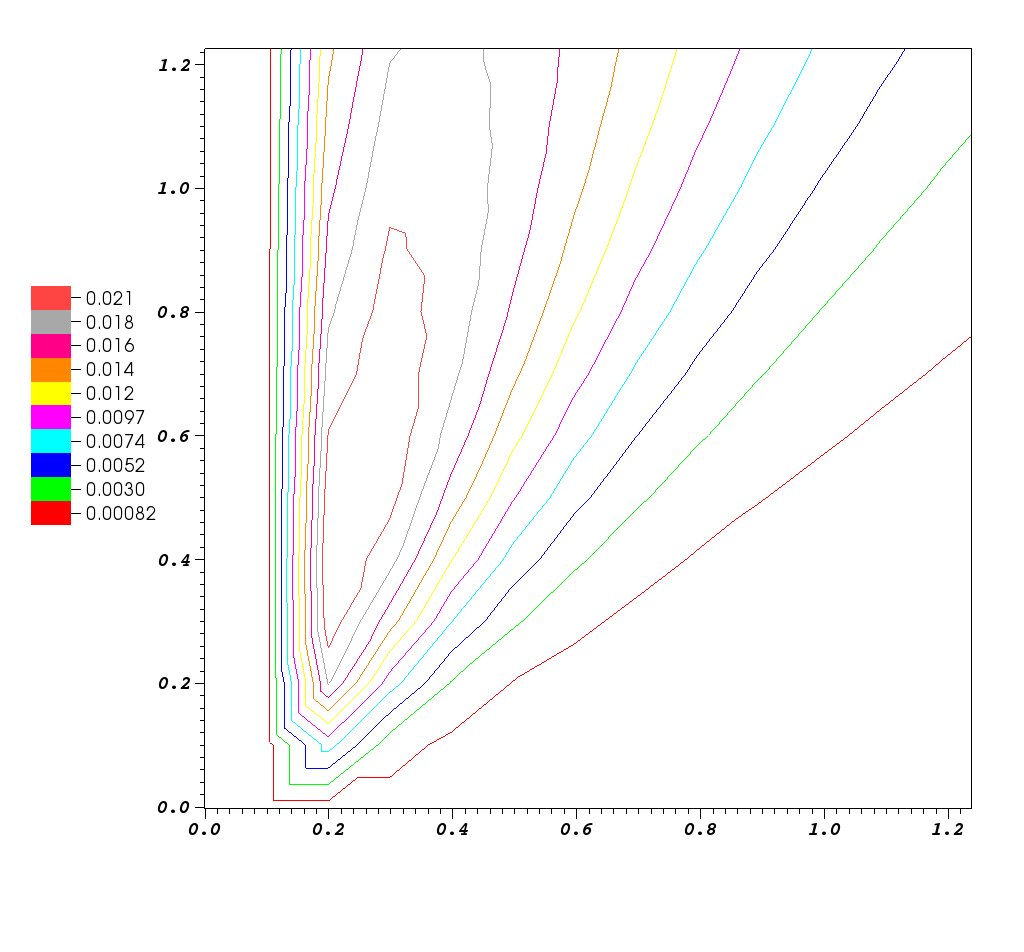}    \\  
  \includegraphics[width=0.4\linewidth]{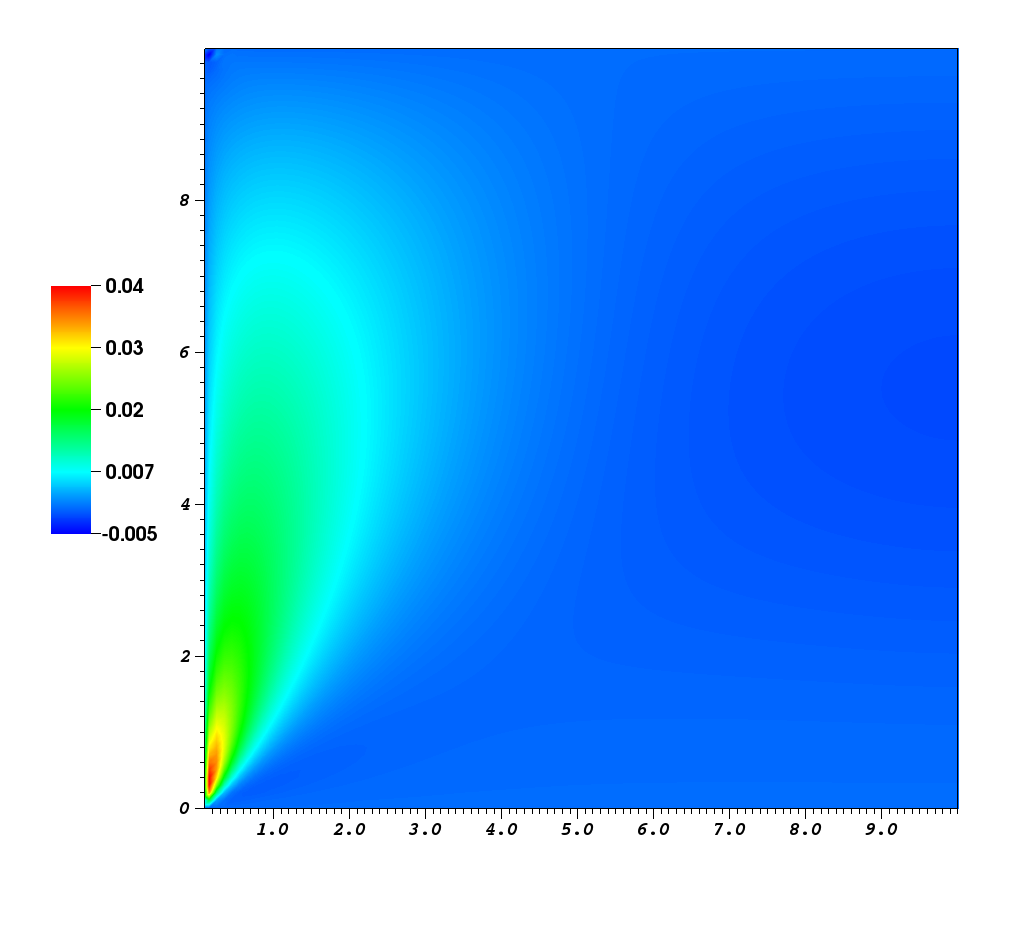}& \includegraphics[width=0.4\linewidth]{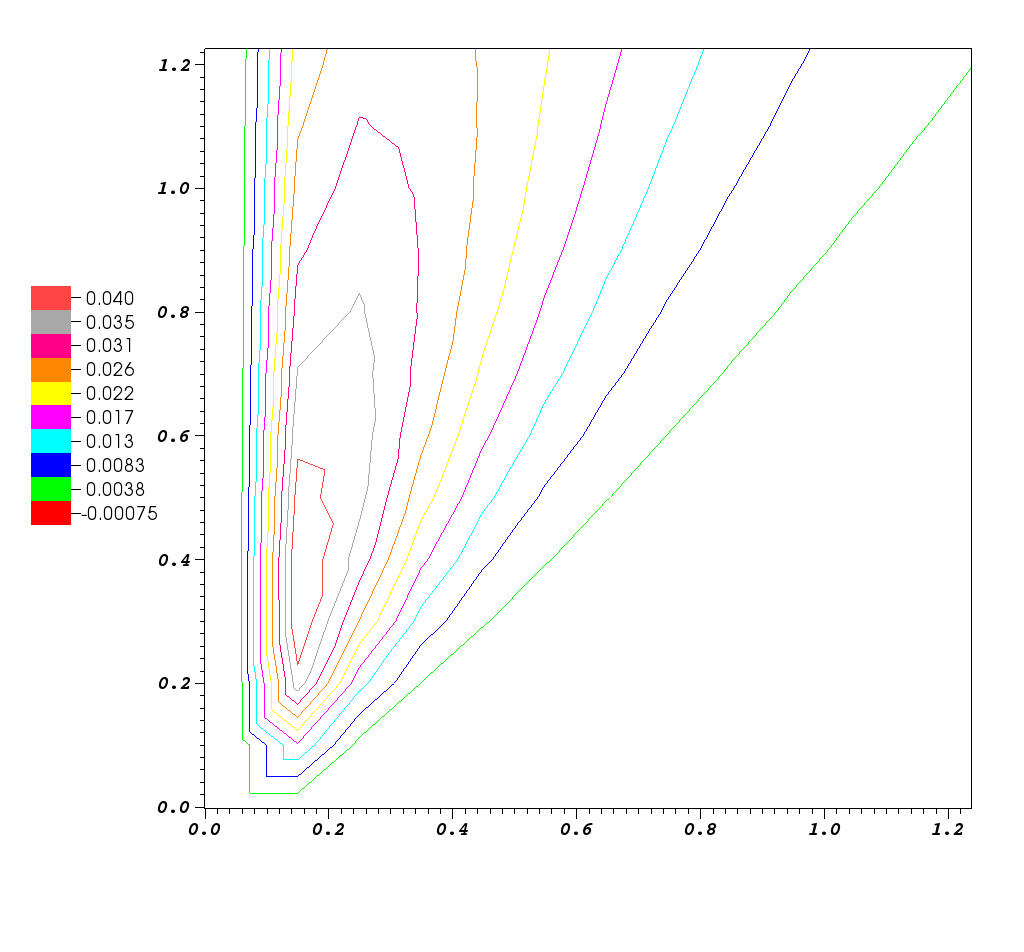} \\
  \includegraphics[width=0.4\linewidth]{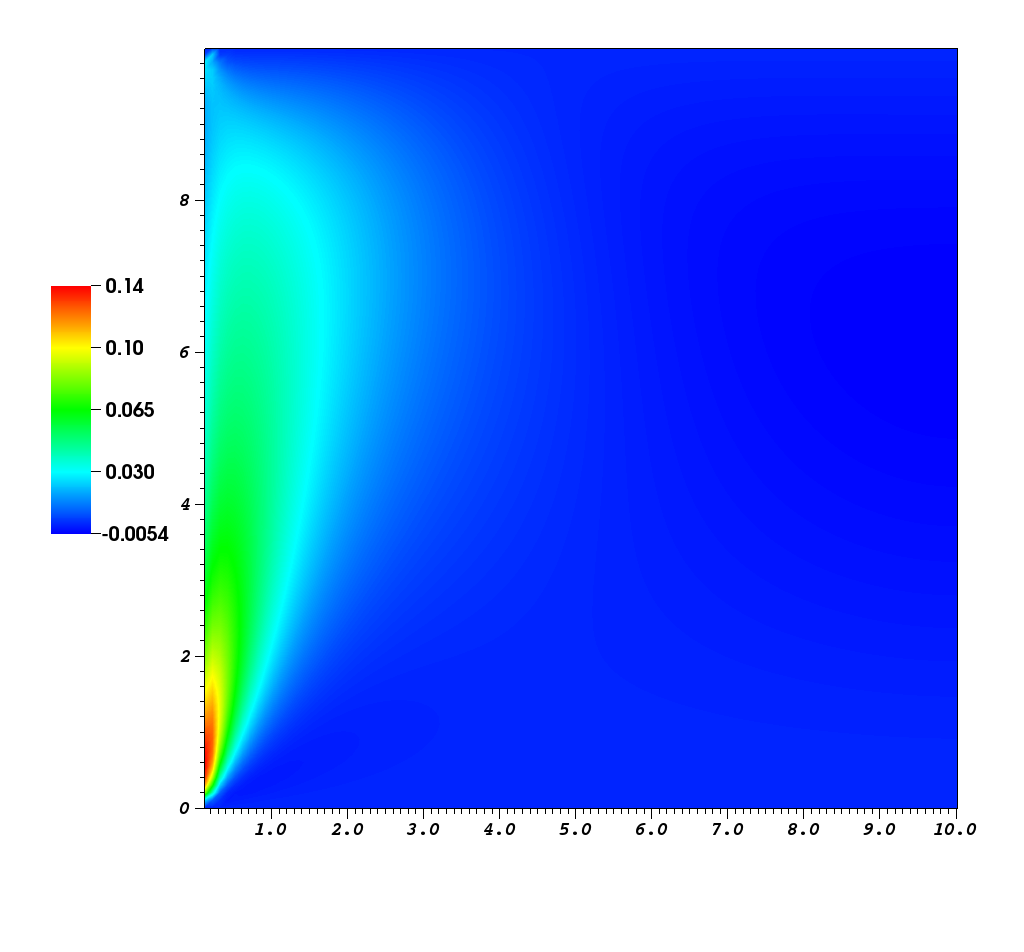}& \includegraphics[width=0.4\linewidth]{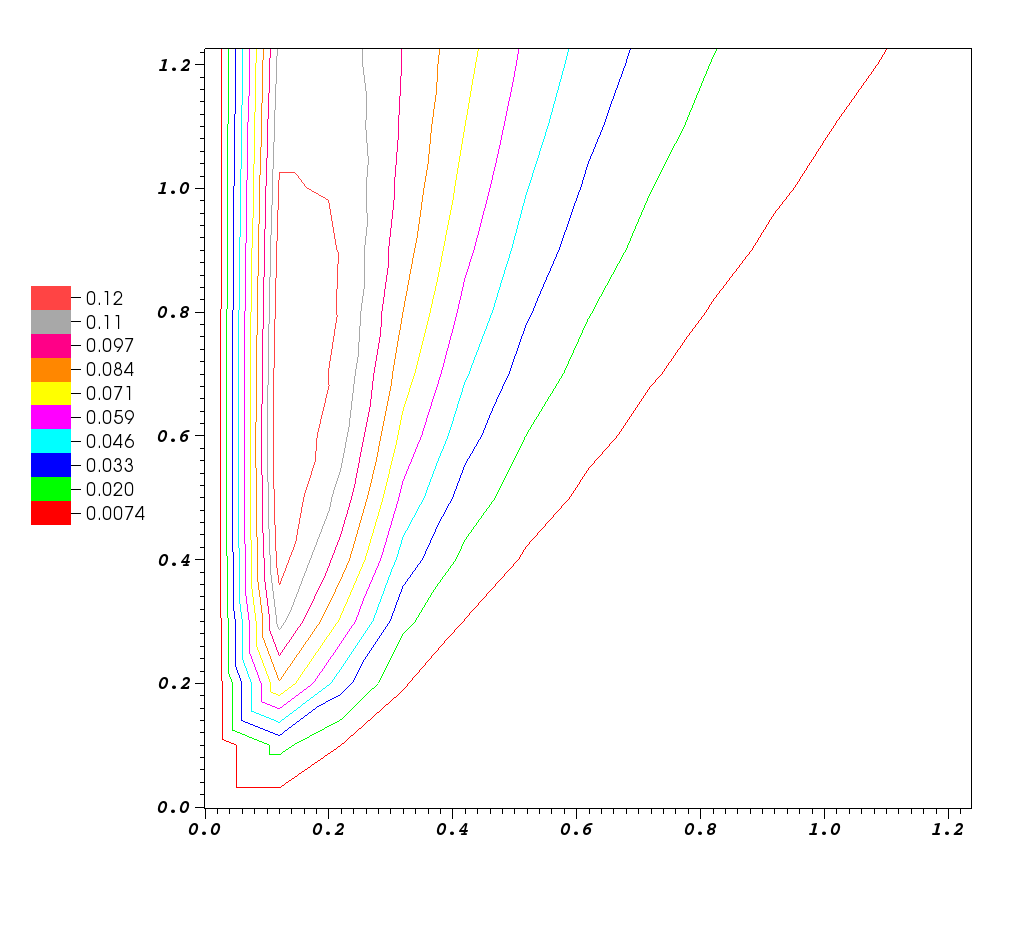}
\end{tabular}
\caption{Contour plots of the velocity component $w$ with different radii $\sigma$. From top to bottom, $\sigma=0.1, 0.05, 0.02$, respectively.  Left column, filled contour plots; right column, contour plots. The parameters are $\nu=0.01$,  $\gamma=0.1 \pi$, $Re=\frac{\gamma}{\nu}=10 \pi$.}
\label{ThrFig}
\end{figure}

In Fig. \ref{meri}, we show the vector plot and the streamline plot of the meridional flow fields (the $u$ and $w$ components) in (a) and (b), respectively,  and the filled contour plot of the pressure in (c) in the region of interest $(r, z) \in [0.1, 1] \times [0, 2]$. The parameters in this simulation are $\sigma=0.1, \nu=0.02, \gamma=\pi$ which defines a rotational Reynolds number $Re=50\pi$.  It demonstrates the phenomenon of boundary layer pumping, a distinguishing process in rotating fluids, cf. \cite{holton2004}. The rotating cylinder imparts to the bulk of the fluids a circular motion. Away from the plane boundary, the radial pressure gradient is approximately balanced by the centrifugal force of the rotating fluid. Near the plane surface, viscosity effect inhibits the fluid motion, giving rise to an unbalanced radial pressure gradient. This pressure gradient creates a radial  inflow towards the axis of the cylinder, which leads to an updraft due to incompressibility.

\begin{figure}[h!]
\centering
\begin{tabular}{ccc}
\includegraphics[width=0.35\linewidth]{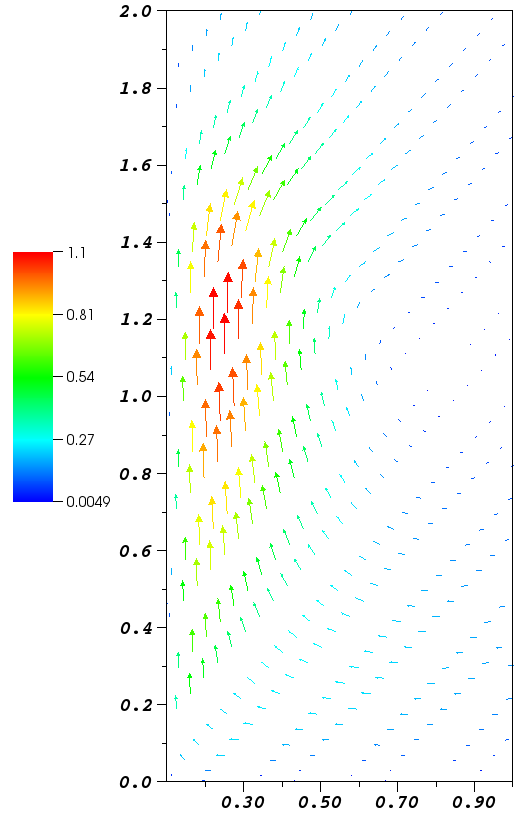}& \includegraphics[width=0.3\linewidth]{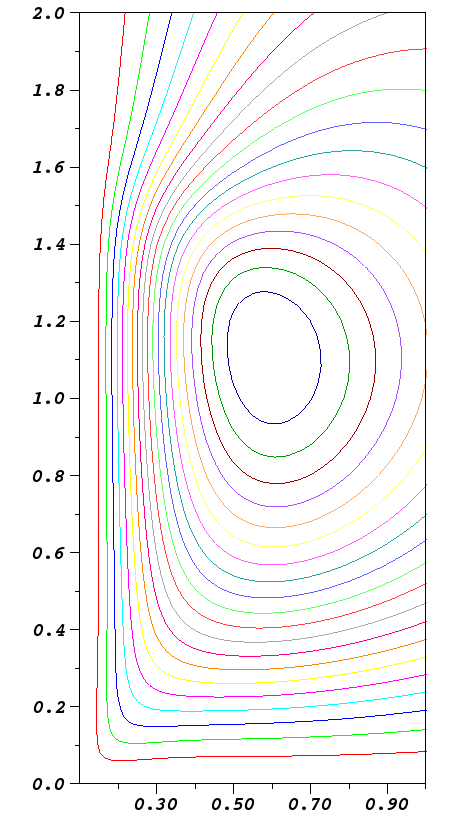} & \includegraphics[width=0.35\linewidth]{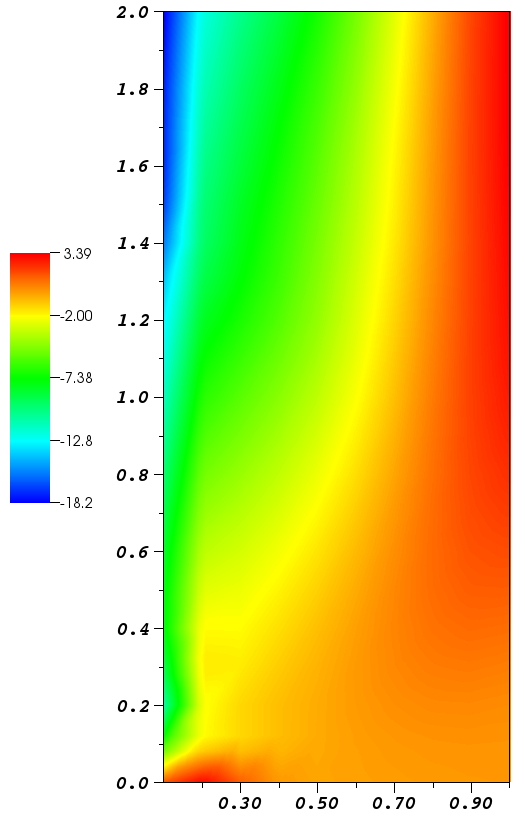}   \\  
(a)  & (b) & (c)
\end{tabular}
\caption{(a) Vector plot of the meridional flow fields ($u$ and $w$); (b) streamline plot of the meridional flow fields; (c) filled contour plot of the pressure. The region shown is $(r, z) \in [0.1, 1] \times [0, 2]$.  The parameters are $\sigma=0.1, \nu=0.02$,  $\gamma=\pi$, $Re=50 \pi$.}
\label{meri}
\end{figure}

In our model, the rotating cylinder of small radius plays the role of a tornado funnel, or the core of a potential line vortex. In the simulations above, the cylinder is rotating at a uniform constant speed. Our result shows that only a single-celled vortex is possible in this setting. In real tornadoes or in tornado vortex chambers, two-celled or multi-celled vortices are observed, cf. \cite{Rotunno2013, Trapp2000} and references therein for the study of vortex breakdown. The downdraft due to vertical pressure gradient is responsible for the creation of multi-celled vortex.  In Serrin's vortex model \cite{Serrin1972}, this vertical pressure gradient is supplied by a pressure boundary condition on the plane--an extra degree of freedom from the model. In our model, we create this downdraft (starting at the top of the cylinder) by varying the rotation speed of the cylinder. Specifically, we impose the following boundary condition for the azimuthal velocity on the cylinder
\begin{eqnarray}\label{Vboun}
v|_{\Gamma_i}=\left\{
\begin{array}{ll}
10 &  \mbox{if $z \leq 2$};\\
1.25\times (10-z) & \mbox{if $ 2 \leq z \leq 10$}.
\end{array}
\right.
\end{eqnarray} 
One sees that the rotational speed of the cylinder decreases when $z \geq 2$. Hence by Bernoulli's law away from the cylinder, the pressure increases with height which leads to the downdraft. A two-celled vortex formed. Fig. \ref{Varyforce} shows the vector plot and streamline plot of the meridional flow under the influence of boundary condition \eqref{Vboun}.

\begin{figure}[h!]
\centering
\begin{tabular}{cc}
\includegraphics[width=0.4\linewidth]{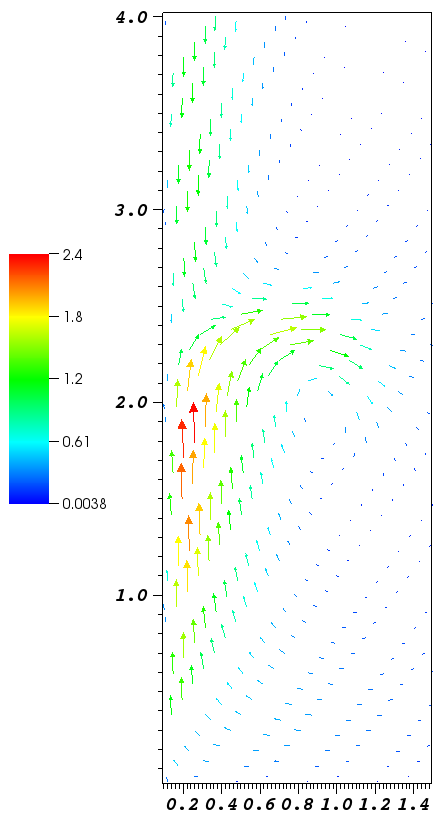} & \includegraphics[width=0.4\linewidth]{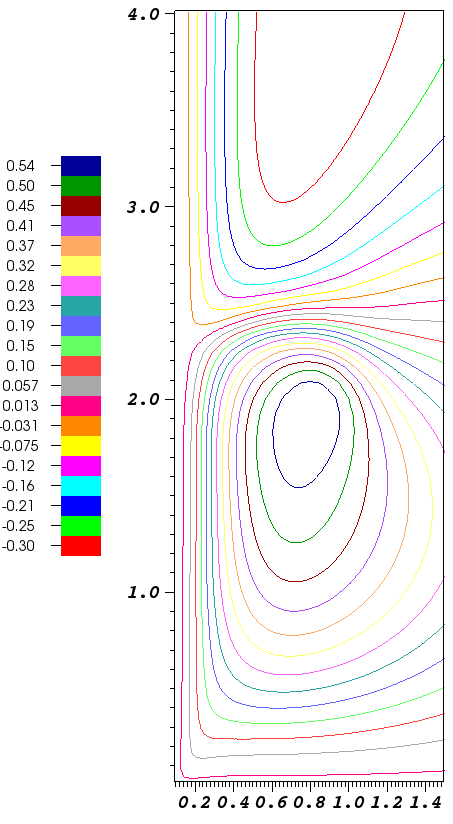}   \\  
(a)  & (b)
\end{tabular}
\caption{(a) Vector plot of the meridional flow fields ($u$ and $w$). (b) streamline plot of the meridional flow fields: positive value, counter clockwise; negative value, clockwise. The boundary condition at the cylinder for $v$ is given by \eqref{Vboun}. The region shown is $(r, z) \in [0.1, 1.5] \times [0, 4]$.  The parameters are $\sigma=0.1, \nu=0.02$,   $Re=100 \pi$.}
\label{Varyforce}
\end{figure}

More complex flow patterns can be generated by tuning the speed of the rotating cylinder. For instance, Fig. \ref{2Varyforce} illustrates the meridional flow fields with the boundary condition for the azimuthal velocity on the cylinder
\begin{eqnarray}\label{2Vboun}
v|_{\Gamma_i}=\left\{
\begin{array}{lll}
4 &  \mbox{if $z \leq 2$};\\
4+2\times (z-2) & \mbox{if $ 2 \leq z \leq 4$};\\
\frac{4}{3}\times (10-z) & \mbox{if $ 4 \leq z \leq 10$}.
\end{array}
\right.
\end{eqnarray} 
Hence the cylinder rotates at constant speed at lower level, then linearly increases to its maximum, and finally decreases to zero. Accordingly, one observes the converging-updraft flow near the corner, then the intensifying updraft due to the drop of the pressure, and finally the downdraft at the upper level. 

\begin{figure}[h!]
\centering
\begin{tabular}{cc}
\includegraphics[width=0.4\linewidth]{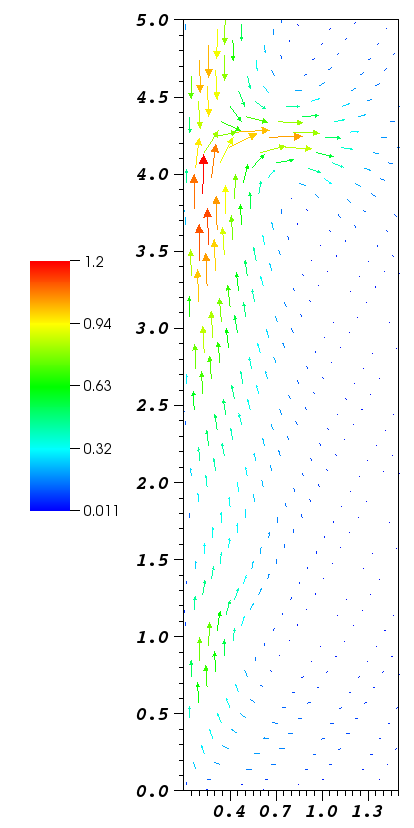} & \includegraphics[width=0.4\linewidth]{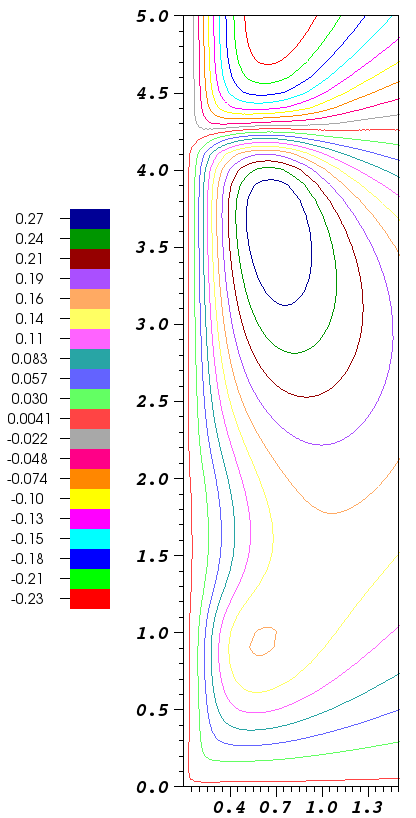}   \\  
(a)  & (b)
\end{tabular}
\caption{(a) Vector plot of the meridional flow fields ($u$ and $w$). (b) streamline plot of the meridional flow fields: positive value, counter clockwise; negative value, clockwise. The boundary condition at the cylinder for $v$ is given by \eqref{2Vboun}. The region shown is $(r, z) \in [0.1, 1.5] \times [0, 5]$.  The parameters are $\sigma=0.1, \nu=0.02$,   $Re=80 \pi$.}
\label{2Varyforce}
\end{figure}

\clearpage

\section*{Acknowledgement}
This work was supported in part by ONR grant N00014-15-1-2385 and by the Research Fund of Indiana University. The authors wish to thank Zhongwei Shen, Joseph Tribbia, and Xiaoming Wang for helpful discussions.

\bibliography{multiphase}{}
\bibliographystyle{plain}

\end{document}